\documentclass[12pt]{article}
\usepackage{amsmath,amsfonts,amsthm,amssymb,bbm,verbatim,graphicx}
\usepackage[left=1in,top=1in,right=1in]{geometry}

\newcommand{\md}{\mathrm{d}}
\newcommand{\e}{\varepsilon}
\newcommand{\unitv}{{\mathbf{e}}}
\newcommand{\dirgraph}{\mathbb{G}}
\def\Z2{\ensuremath{\mathbb{Z}^2}}
\def\R2{\ensuremath{\mathbb{R}^2}}
\def\E2{\ensuremath{\mathcal{E}^2}}

\def\passage{\ensuremath{\mathbb{P}}}
\newtheorem{thm}{Theorem}[section]
\newtheorem{prop}[thm]{Proposition}
\newtheorem{lem}[thm]{Lemma}
\newtheorem{cor}[thm]{Corollary}
\newtheorem{clam}[thm]{Claim}
\newtheorem{rem}[thm]{Remark}
\newtheorem{df}[thm]{Definition}

\numberwithin{equation}{section}

\begin{document}

\title{Busemann functions and infinite geodesics in two-dimensional first-passage percolation}
\date{\today}
\author{Michael Damron \thanks{M. D. is supported by an NSF postdoctoral fellowship and NSF grants DMS-0901534 and DMS-1007626.} \\ \small Department of Mathematics \\ \small Princeton University \and Jack Hanson \thanks{J. H. is supported by an NSF graduate fellowship and NSF grant PHY-1104596.} \\ \small Department of Physics \\ \small Princeton University}

\maketitle

\begin{abstract}
We study first-passage percolation on $\mathbb{Z}^2$, where the edge weights are given by a translation-ergodic distribution, addressing questions related to existence and coalescence of infinite geodesics. Some of these were studied in the late 90's by C. Newman and collaborators under strong assumptions on the limiting shape and weight distribution. In this paper we develop a framework for working with distributional limits of Busemann functions and use it to prove forms of Newman's results under minimal assumptions. For instance, we show a form of coalescence of long finite geodesics in any deterministic direction. We also introduce a purely directional condition which replaces Newman's global curvature condition and whose assumption we show implies the existence of directional geodesics. Without this condition, we prove existence of infinite geodesics which are directed in sectors. Last, we analyze distributional limits of geodesic graphs, proving almost-sure coalescence and nonexistence of infinite backward paths. This result relates to the conjecture of nonexistence of ``bigeodesics.''
\end{abstract}

\tableofcontents

\section{Introduction}

First-passage percolation (FPP) was introduced by Hammersley and Welsh \cite{HW} as a model for fluid flow through a porous medium. However, it has since developed into a field of its own, serving for instance as a model for growing interfaces (see \cite{KS} and connections to other models \cite{HH}) and competing infections (see \cite{BS, DH, GM, HP, Hoffman1}). For a survey of recent results, see \cite{GK}.

We consider FPP on $(\Z2,\E2),$ the two-dimensional square lattice. $\mathbb{P}$ will denote a probability measure on the space $\Omega = \mathbb{R}^{\E2}$ (satisfying some conditions outlined in the next section). An element $\omega \in \Omega$ represents an edge-weight configuration; the passage time across the edge $e$ is denoted $\omega_e = \omega(e)$. The passage time between two sites $x,y$ will be called
\[
\tau(x,y) = \inf_{\gamma:x \to y} \tau(\gamma)\ ,
\]
where the infimum is over all (finite) lattice paths from $x$ to $y$ and $\tau(\gamma) = \sum_{e \in \gamma} \omega_e$.

In this paper we study geodesics, (typically self-avoiding) paths in $\mathbb{Z}^2$ which are everywhere time-minimizing. Precisely, define a finite geodesic from $x$ to $y$ to be a finite lattice path $\gamma$ from $x$ to $y$ such that $\tau(\gamma) = \tau(x,y)$. Define an infinite geodesic to be an infinite path such that each finite subpath is a finite geodesic. In the mid 90's, Newman \cite{Newman95} and Licea-Newman \cite{LN}, along with Wehr \cite{Wehr} began the rigorous study of infinite geodesics. This was in part motivated by connections between ``bigeodesics'' in FPP and ground states of disordered ferromagnetic spin models \cite{FLN, newmanbook}. The main questions involve existence of infinite geodesics with asymptotic directions, uniqueness and coalescence of such geodesics, and absence of bigeodesics. After considerable progress on lattice FPP, Howard and Newman gave an essentially complete description for a continuum variant, called Euclidean FPP \cite{HN}.

The main theorems proved to date require heavy assumptions on the model, for instance strong moment bounds and so-called curvature inequalities (the establishment of which provides a major open problem in FPP). The main goals of this paper are to prove versions of the current geodesic theorems under minimal assumptions necessary to guarantee their validity. Because the methods of Newman and collaborators involve curvature bounds and concentration inequalities (the latter of which cannot hold under low moment assumptions), we are forced to develop completely new techniques. 

Our analysis centers on Busemann functions, which were used and analyzed in papers of Hoffman \cite{Hoffman1, Hoffman2}. His work was one of the first (along with Garet-Marchand \cite{GM}) to assert existence of multiple disjoint infinite geodesics under general assumptions, finding at least four almost surely. The methods are notable in their ability to extract any information without knowing the existence of limits for Busemann functions. Indeed, proving the existence of such limits, corresponding to
\[
\lim_{n \to \infty} \left[ \tau(x,x_n) - \tau(y,x_n) \right]
\]
for fixed $x,y$ and a deterministic sequence of vertices $(x_n)$ growing to infinity along a ray, provides a major open problem and appears to be an impediment to further analysis of geodesics in the model. Incidentally, in an effort to describe the microstructure of the limiting shape for the model, Newman \cite{Newman95} was able to show that under strong assumptions, this limit exists in Lebesgue-almost every direction. 

One main aim of the present paper is to develop a framework to overcome the existence of the above limit. We will analyze distributional limits of Busemann functions and relate these back to the first-passage model. The relationship between Busemann functions and geodesics will be preserved in the limit and will provide information about directional geodesics, coalescence, and the structure of geodesic graphs, the latter of which gives nonexistence of certain types of bigeodesics.

\subsection{Main results}

We will make one of two main assumptions on the passage time distribution. These relate to the degree of independence in the model.  The first deals with i.i.d. passage times:
\begin{enumerate}
\item[{\bf A1}] $\mathbb{P}$ is a product measure whose common distribution satisfies the criterion of Cox and Durrett \cite{coxdurrett}: if $e_1, \ldots, e_4$ are the four edges touching the origin,
\begin{equation}\label{eq: jackisherenow}
\mathbb{E} \left[ \min_{i=1,\ldots, 4} \omega_{e_i} \right]^2 < \infty\ .
\end{equation}
Furthermore we assume $\mathbb{P}(\omega_e= 0) < p_c=1/2$, the bond percolation threshold for $\mathbb{Z}^2$.
\end{enumerate}
Condition \eqref{eq: jackisherenow} is implied by, for example, the assumption $\mathbb{E} \omega_e < \infty$.

The other assumption is on distributions that are only translation-invariant. Condition (d) below deals with the limit shape, which is defined in the next paragraph.
\begin{enumerate}
\item[{\bf A2}] $\mathbb{P}$ is a measure satisfying the conditions of Hoffman \cite{Hoffman2}:
\begin{enumerate}
\item $\mathbb{P}$ is ergodic with respect to translations of $\mathbb{Z}^2$;
\item $\mathbb{P}$ has all the symmetries of $\mathbb{Z}^2$;
\item $\mathbb{E} \omega_e^{2+\e} < \infty$ for some $\e>0$;
\item the limit shape for $\mathbb{P}$ is bounded.
\end{enumerate}
\end{enumerate}
Some of the conditions here can be weakened. For instance, the $2+\e$ moment condition can be replaced with a condition of a finite Lorentz-type norm; see \cite{boivin} for details.

In each of these settings, a ``shape theorem'' has been proved \cite{boivin, coxdurrett} for the set of sites accessible from 0 in time $t$.  For $x,y \in \mathbb{R}^2$ we set $\tau(x,y) = \tau(\tilde x, \tilde y)$, where $\tilde x$ and $\tilde y$ are the unique points in $\mathbb{Z}^2$ such that $x \in \tilde x + [-1/2,1/2)^2$ and $y \in \tilde y + [-1/2,1/2)^2$. For any $t\geq 0$ write $B(t)$ for the set of $x$ in $\mathbb{R}^2$ such that $\tau(0,x) \leq t$ and $B(t)/t = \{x/t : x \in B(t)\}$.  There exists a deterministic compact convex set $\mathcal{B}$, symmetric about the axes and with nonempty interior such that for each $\e>0$,
\[
\mathbb{P}\left( (1-\e) \mathcal{B} \subseteq B(t)/t \subseteq (1+\e) \mathcal{B} \text{ for all large } t\right) = 1\ .
\]
The statement that $\mathcal{B}$ has nonempty interior is not explicitly proved in \cite{boivin} but follows from the maximal lemma stated there.

\subsubsection{Directional results}

Our first results deal with asymptotic directions for infinite geodesics. Much is known about such questions under various strong assumptions (for instance uniformly positive curvature of $\mathcal{B}$, exponential moments for $\mathbb{P}$; see Section~\ref{sec: global} for a more precise discussion). However, under only {\bf A1} or {\bf A2}, very little is known. After initial results by H\"aggstr\"om-Pemantle \cite{HP}, Garet-Marchand \cite{GM} and Hoffman \cite{Hoffman1}, it was proved by Hoffman \cite{Hoffman2} that under {\bf A2}, there exist at least 4 infinite geodesics that are pairwise disjoint almost surely. Nothing is known about the directions of the geodesics; for instance, Hoffman's results do not rule out the case in which the geodesics spiral around the origin.

Below we will show that under {\bf A1} or {\bf A2} there are geodesics that are asymptotically directed in sectors of aperture no bigger than $\pi/2$. Under a certain directional condition on the boundary of the limit shape (see Corollary~\ref{cor: exposed}) we show existence of geodesics with asymptotic direction. To our knowledge, the only work of this type so far \cite[Theorem~2.1]{Newman95} requires a global curvature assumption to show the existence of geodesics in even one direction.

To describe the results, we endow $[0,2\pi)$ with the distance of $S^1$: say that $dist(\theta_1,\theta_2) < r$ if there exists an integer $m$ such that $|\theta_1-\theta_2 - 2\pi m| < r$. For $\Theta \subseteq [0,2\pi)$ we say that a path $\gamma = x_0, x_1, \ldots$ is {\it asymptotically directed in} $\Theta$ if for each $\e>0$, $\arg x_k \in \Theta_\e \text{ for all large } k$, where $\Theta_\e = \{\theta: dist(\theta,\phi) < \e \text{ for some } \phi \in \Theta\}$. For $\theta \in [0,2\pi)$, write $v_\theta$ for the unique point of $\partial \mathcal{B}$ with argument $\theta$. Recall that a supporting line $L$ for $\mathcal{B}$ at $v_\theta$ is one that touches $\mathcal{B}$ at $v_\theta$ such that $\mathcal{B}$ lies on one side of $L$. If $\theta$ is an angle such that $\partial \mathcal{B}$ is differentiable at $v_\theta$ (and therefore has a unique supporting line $L_\theta$ (the tangent line) at this point), we define an interval of angles $I_\theta$:
\begin{equation}\label{eq: last_night}
I_\theta = \{\theta' : v_{\theta'} \in L_\theta\}\ .
\end{equation}

\begin{thm}\label{thm: sectors}
Assume either {\bf A1} or {\bf A2}. If $\partial \mathcal{B}$ is differentiable at $v_\theta$, then with probability one there is an infinite geodesic containing the origin which is asymptotically directed in $I_\theta$.
\end{thm}

The meaning of the theorem is that there is a measurable set $\mathcal{A}$ with $\mathbb{P}(\mathcal{A})=1$ such that if $\omega \in \mathcal{A}$, there is an infinite geodesic containing the origin in $\omega$ which is asymptotically directed in $I_\theta$. This also applies to any result we state with the phrases ``with probability one there is an infinite geodesic'' or ``with probability one there is a collection of geodesics.''

We now state two corollaries. A point $x \in \partial \mathcal{B}$ is {\it exposed} if there is a line that touches $\mathcal{B}$ only at $x$.

\begin{cor}\label{cor: exposed}
Assume either {\bf A1} or {\bf A2}. Suppose that $v_\theta$ is an exposed point of differentiability of $\partial \mathcal{B}$. With probability one there exists an infinite geodesic containing the origin with asymptotic direction $\theta$.
\end{cor}
\begin{proof}
Apply Theorem~\ref{thm: sectors}, noting that $I_\theta = \{\theta\}$.
\end{proof}

In the next corollary we show that there are infinite geodesics asymptotically directed in certain sectors. Because the limit shape is convex and compact, it has at least 4 extreme points. Angles corresponding to the arcs connecting these points can serve as the sectors.

\begin{cor}\label{cor: extreme}
Assume either {\bf A1} or {\bf A2}. Let $\theta_1\neq \theta_2$ be such that $v_{\theta_1}$ and $v_{\theta_2}$ are extreme points of $\mathcal{B}$. If $\Theta$ is the set of angles corresponding to some arc of $\partial \mathcal{B}$ connecting $v_{\theta_1}$ to $v_{\theta_2}$, then with probability one there exists an infinite geodesic containing the origin which is asymptotically directed in $\Theta$.
\end{cor}

\begin{proof}

Choose $\theta_3 \in \Theta$ such that $\theta_1 \neq \theta_3 \neq \theta_2$ and $\mathcal{B}$ has a unique supporting line $L_{\theta_3}$ at $v_{\theta_3}$ (this is possible since the boundary is differentiable almost everywhere). 
Let $C$ be the closed arc of $\partial \mathcal{B}$ from $v_{\theta_1}$ to $v_{\theta_2}$ that contains $v_{\theta_3}$ and write $D$ for its open complementary arc. We claim $D \subseteq I_{\theta_3}^c$. This will prove the corollary after applying Theorem~\ref{thm: sectors} with $\theta = \theta_3$.

For a contradiction, suppose that $L_{\theta_3}$ intersects $D$ at some point $v_{\phi}$ and write $S$ for the segment of $L_{\theta_3}$ between $v_{\theta_3}$ and $v_\phi$. Since $L_{\theta_3}$ is a supporting line, the set $\mathcal{B}$ lies entirely on one side of it. On the other hand, since $\mathcal{B}$ is convex and $v_{\theta_3}, v_\phi \in \mathcal{B}$, $S \subseteq \mathcal{B}$. Therefore $S \subseteq \partial \mathcal{B}$ and must be an arc of the boundary. It follows that one of $v_{\theta_1}$ or $v_{\theta_2}$ is in the interior of $S$, contradicting the fact that these are extreme points of $\mathcal{B}$.
\end{proof}

\begin{rem}
If $\mathbb{P}$ is a product measure with $\mathbb{P}(\omega_e=1) = \vec p_c \text{ and } \mathbb{P}(\omega_e <1) = 0$, where $\vec p_c$ is the critical value for directed percolation, \cite[Theorem~1]{AD11} implies that $(1/2,1/2)$ is an exposed point of differentiability of $\mathcal{B}$. Corollary~\ref{cor: exposed} then gives a geodesic in the direction $\pi/4$. Though all points of $\partial \mathcal{B}$  (for all measures not in the class of Durett-Liggett \cite{durrettliggett}) should be exposed points of differentiability, this is the only proven example.
\end{rem}

\begin{rem}
From \cite[Theorem~1.3]{HM}, for any compact convex set $\mathcal{C}$ which is symmetric about the axes with nonempty interior, there is a measure $\mathbb{P}$ satisfying {\bf A2} (in fact, with bounded passage times) which has $\mathcal{C}$ as a limit shape. Taking $\mathcal{C}$ to be a Euclidean disk shows that there exist measures for which the corresponding model obeys the statement of Corollary~\ref{cor: exposed} in any deterministic direction $\theta$.
\end{rem}

\subsubsection{Global results}\label{sec: global}

In this section we use the terminology of Newman \cite{Newman95}. Call $\theta$ a {\it direction of curvature} if there is a Euclidean ball $B_\theta$ with some center and radius such that $\mathcal{B} \subseteq B_\theta$ and $\partial B_\theta \cap \mathcal{B} = \{v_\theta\}$. We say that $\mathcal{B}$ has {\it uniformly positive curvature} if each direction is a direction of curvature and there exists $M< \infty$ such that the radius of $B_\theta$ is bounded by $M$ for all $\theta$.

In \cite[Theorem~2.1]{Newman95}, Newman has shown that under the assumptions (a) $\mathbb{P}$ is a product measure with $\mathbb{E} e^{\beta \omega_e}< \infty$ for some $\beta>0$, (b) the limit shape $\mathcal{B}$ has uniformly positive curvature and (c) $\omega_e$ is a continuous variable, two things are true with probability one.
\begin{enumerate}
\item For each $\theta \in [0,2\pi)$, there is an infinite geodesic with asymptotic direction $\theta$.
\item Every infinite geodesic has an asymptotic direction.
\end{enumerate}
As far as we know, there has been no weakening of these assumptions.

Below we improve on Newman's theorem. We first reduce the moment assumption on $\mathbb{P}$ to that of {\bf A1}. Next we extend the theorem to non-i.i.d. measures. Newman's proof uses concentration inequalities of Kesten \cite{Kesten} and Alexander \cite{Alexander}, which require exponential moments on the distribution (and certainly independence). So to weaken the moment assumptions we need to use a completely different method, involving Busemann functions instead.

To state the theorem, we make slightly stronger hypotheses:
\begin{enumerate}
\item[{\bf A1'}] $\mathbb{P}$ satisfies {\bf A1} and the common distribution of $\omega_e$ is continuous.
\item[{\bf A2'}] $\mathbb{P}$ satisfies {\bf A2} and $\mathbb{P}$ has unique passage times.
\end{enumerate}
The phrase ``unique passage times'' means that for all paths $\gamma$ and $\gamma'$ with distinct edge sets, $\mathbb{P}(\tau(\gamma)=\tau(\gamma'))=0$.

\begin{thm}\label{thm: newman}
Assume either {\bf A1'} or {\bf A2'} and that $\mathcal{B}$ has uniformly positive curvature.
\begin{enumerate}
\item With $\mathbb{P}$-probability one, for each $\theta$ there is an infinite geodesic with direction $\theta$.
\item With $\mathbb{P}$-probability one, every infinite geodesic has a direction.
\end{enumerate}
\end{thm}

The same method of proof shows the following. 

\begin{cor}\label{cor: newman2}
Assume either {\bf A1'} or {\bf A2'} and suppose $v_\theta$ is an exposed point of differentiability of $\partial \mathcal{B}$ for all $\theta$. Then the conclusions of Theorem~\ref{thm: newman} hold.
\end{cor}

\begin{rem}
The proofs of the above two results only require that the set of extreme points of $\mathcal{B}$ is dense in $\partial \mathcal{B}$. In fact, a similar result holds for a sector in which extreme points of $\mathcal{B}$ are dense in the arc corresponding to this sector.
\end{rem}

\subsubsection{Coalescence for geodesics}\label{subsec: coalesce}

In this section we describe results for coalescence of infinite geodesics. For this we need some notation. For $S\subseteq \mathbb{R}^2$ define the point-to-set passage time 
\[
\tau(x,S) = \inf_{y \in S} \tau(x,y) \text{ for } x \in \mathbb{R}^2\ .
\]
By the subadditivity property $\tau(x,y) \leq \tau(x,z) + \tau(z,y)$ we find
\begin{equation}\label{eq: subadditivity}
\tau(x,S) \leq \tau(x,y) + \tau(y,S) \text{ for } x,y \in \mathbb{R}^2\ .
\end{equation}
A path $\gamma$ from a point $x \in \Z2$ to a point in 
\begin{equation}\label{eq: hatS}
\hat S = \{y \in \Z2 : y + [-1/2,1/2)^2 \cap S \neq \varnothing\}
\end{equation}
is called a {\it geodesic from $x$ to $S$} if $\tau(\gamma) = \tau(x,S)$. Under assumptions {\bf A1} or {\bf A2}, one can argue from the shape theorem and boundedness of the limit shape that a geodesic from $x$ to $S$ exists $\mathbb{P}$-almost surely. However, it need not be unique.  In the case, though, that we assume {\bf A1'} or {\bf A2'}, there is almost surely exactly one geodesic from $x$ to $S$. Note that if $\gamma$ is a geodesic from $x$ to $S$ and $y\in \gamma$, then the piece of $\gamma$ from $x$ to $y$ is a geodesic from $x$ to $y$ and the piece of $\gamma$ from $y$ to $S$ is a geodesic from $y$ to $S$.

The set $S$ gives a directed geodesic graph $\mathbb{G}_S = \mathbb{G}_S(\omega)$: $\langle x,y \rangle$ is an edge of $\mathbb{G}_S$ if it is in some geodesic from a point to $S$ and $\tau(x,S) \geq \tau(y,S)$ (we will explain more about this graph in Section~\ref{subsec: GG}). We say that a sequence of directed graphs $G_n = (\mathbb{Z}^2,E_n)$ converges to a directed graph $G=(\mathbb{Z}^2,E)$ if each edge $\langle x,y \rangle$ is in only finitely many of the symmetric differences $E_n \Delta E$. If $x$ and $y$ are vertices of a directed graph $G$, write $x \to y$ if there is a directed path from $x$ to $y$ in $G$. Last, we say that two infinite directed paths $\Gamma$ and $\Gamma'$ {\it coalesce} if their (edge) symmetric difference is finite.

For the main theorems on coalescence we need an extra assumption in the case {\bf A2'}. It allows us to apply ``edge modification'' arguments. Write $\omega = (\omega_e,\check{\omega})$, where $\check{\omega}_f = (\omega)_{f\neq e}$. 

\begin{df}
We say that $\passage$ has the {\it upward finite energy property} if for each $\lambda > 0$ such that $\mathbb{P}(\omega_e \geq \lambda)>0$,
		\begin{equation}
		\label{finite_energy_def}
		\passage\left(\omega_e \geq \lambda \, \big|\, \check{\omega} \right) >0 \quad \text{almost surely}\ .
		\end{equation}
\end{df}
Note that if $\mathbb{P}$ is a product measure, it has the upward finite energy property.

\begin{thm}\label{thm: random_hyperplanes}
Assume either {\bf A1'} or both {\bf A2'} and the upward finite energy property. Let $v \in \mathbb{R}^2$ be any nonzero vector and for $\beta \in \mathbb{R}$ define
\[
L_\beta(v) = \{y \in \mathbb{R}^2 : y \cdot v = \beta\}\ .
\]
There exists an event $\mathcal{A}$ with $\mathbb{P}(\mathcal{A}) = 1$ such that for each $\omega \in \mathcal{A}$, the following holds. There exists an ($\omega$-dependent) increasing sequence $(\alpha_k)$ of real numbers with $\alpha_k \to \infty$ such that $\mathbb{G}_{L_{\alpha_k}(v)}(\omega) \to G(\omega)$, a directed graph with the following properties.
\begin{enumerate}
\item Viewed as an undirected graph, $G$ has no circuits.
\item Each $x \in \mathbb{Z}^2$ has out-degree 1 in $G$.
\item (All geodesics coalesce.) Write $\Gamma_x$ for the unique infinite path in $G$ from $x$. If $x,y \in \mathbb{Z}^2$ then $\Gamma_x$ and $\Gamma_y$ coalesce.
\item (Backward clusters are finite.) For all $x \in \mathbb{Z}^2$, the set $\{y \in \mathbb{Z}^2: y \to x \text{ in } G\}$ is finite.
\end{enumerate} 
\end{thm}

Our last theorem deals with coalescence and asymptotic directions. Before stating it, we discuss some previous results. In 1995, Licea and Newman \cite{LN} proved that given $\theta \in [0,2\pi)$, all directional geodesics almost surely coalesce except in some deterministic (Lebesgue-null) set $D \subseteq [0,2\pi)$. Specifically they showed that under the assumptions (a) $\mathbb{P}$ is a product measure whose one-dimensional marginals are continuous with finite exponential moments and (b) uniformly positive curvature of $\mathcal{B}$,
\begin{equation}\label{eq: liceanewman}
\text{there exists }D \subseteq [0,2\pi) \text{ with Lebesgue measure zero such that if } \theta \in [0,2\pi) \setminus D\ ,
\end{equation}
\begin{enumerate}
\item almost surely, there exists a collection of infinite geodesics $\{\gamma_x : x \in \mathbb{Z}^2\}$ such that each $\gamma_x$ has asymptotic direction $\theta$ and for all $x,y$, the paths $\gamma_x$ and $\gamma_y$ coalesce and 
\item almost surely, for each $x$, there is a unique infinite geodesic containing $x$ with asymptotic direction $\theta$.
\end{enumerate}
Since \cite{LN} it has been an open problem to show that $D$ can be taken to be empty. Zerner \cite[Theorem~1.5]{newmanbook} proved that $D$ can be taken to be countable. In a related exactly solvable model (directed last-passage percolation, using exponential weights on sites), Coupier has proved \cite[Theorem~1(3)]{coupier}, building on work of Ferrari-Pimentel \cite{FP}, that $D$ can be taken to be empty. These results rely on a mapping to the TASEP particle system.

In part 2 of the next theorem, we improve on \eqref{eq: liceanewman} in the general case. The result reduces the set $D$ to be empty for existence of coalescing geodesics (item 1 above). It however does not address uniqueness. We reduced the moment condition of \cite{LN}, extended to non-i.i.d. measures and replaced the global curvature assumption with a directional condition. Without this condition, part 3 gives the existence of coalescing geodesics directed in sectors. For the statement, recall the definition of $I_\theta$ in \eqref{eq: last_night}.

\begin{thm}\label{thm: exceptional_set}
Assume either {\bf A1'} or both {\bf A2'} and the upward finite energy property. Let $\theta \in [0,2\pi)$.
\begin{enumerate}
\item If $\partial \mathcal{B}$ is differentiable at $v_\theta$ then with probability one there exists a collection $\{\gamma_x:x \in \mathbb{Z}^2\}$ of infinite geodesics in $\omega$ such that 
\begin{enumerate}
\item each $x$ is a vertex of $\gamma_x$;
\item each $\gamma_x$ is asymptotically directed in $I_\theta$;
\item for all $x,y \in \mathbb{Z}^2$, $\gamma_x$ and $\gamma_y$ coalesce and
\item each $x$ is on $\gamma_y$ for only finitely many $y$.
\end{enumerate}
\item If $v_\theta$ is an exposed point of differentiability of $\mathcal{B}$ then the above geodesics all have asymptotic direction $\theta$.
\item Suppose $\theta_1 \neq \theta_2$ are such that $v_{\theta_1}$ and $v_{\theta_2}$ are extreme points of $\mathcal{B}$. If $\Theta$ is the set of angles corresponding to some arc of $\partial \mathcal{B}$ connecting $v_{\theta_1}$ to $v_{\theta_2}$ then the above geodesics can be taken to be asymptotically directed in $\Theta$.
\end{enumerate}
\end{thm}

Theorems~\ref{thm: random_hyperplanes} and \ref{thm: exceptional_set} follow from a stronger result. In Sections~\ref{sec: GG} and \ref{sec: coalesceG}, we prove that any subsequential limit $\mu$ defined as in Section~\ref{sec: mudef} is supported on geodesic graphs with properties 1-4 of Theorem~\ref{thm: random_hyperplanes}.

\begin{rem}
The finiteness of backward clusters in the graphs produced in the previous two theorems (see item 4 of the first and item 1(d) of the second) is related to nonexistence of bigeodesics. It shows that when constructing infinite geodesics using a certain limiting procedure, it is impossible for doubly infinite paths to arise.
\end{rem}

\subsection{Notation}

We denote the standard orthonormal basis vectors for \R2 by $\unitv_1$ and $\unitv_2.$  The translation operators $T_{\unitv_i},~ i = 1,2$
act on a configuration $\omega$ as follows: 
$\left(T_{\unitv_i} (\omega) \right)_{e'} = \omega_{e'-\unitv_i}.$ Under any of the assumptions laid out above, the measure $\mathbb{P}$ is invariant under these translations. Furthermore the passage times have a certain translation-covariance: for $i=1,2$,
\begin{equation}\label{eq: transcov}
\tau(x,S)(T_{\unitv_i}\omega) = \tau(x-\unitv_i,S-\unitv_i)(\omega)\ ,
\end{equation}
where $S-\unitv_i = \{ x-\unitv_i : x \in S\}$.

We shall need a function $g:\mathbb{R}^2 \to \mathbb{R}$ which describes the limiting shape $\mathcal{B}$. It is the norm whose closed unit ball is $\mathcal{B}$. There are many ways to define it; for instance one can use $g(x) = \inf\{\lambda > 0:  x/\lambda \in \mathcal{B}\}$. It follows from the shape theorem that under {\bf A1} or {\bf A2},
\[
\lim_{n \to \infty} \tau(0,nx)/n = g(x) \text{ for all } x \in \mathbb{R}^2, ~\mathbb{P}\text{-almost surely}\ .
\]
Furthermore, there is convergence in $L^1$:
\[
\lim_{n \to \infty} \mathbb{E}\tau(0,nx)/n = g(x) \text{ for all } x \in \mathbb{R}^2\ .
\]
In the case of {\bf A1} this follows from \cite[Lemma~3.2]{coxdurrett} and under {\bf A2} it can be derived from the shape theorem and \cite[Lemma~2.6]{Hoffman2} (the reader can also see a derivation in the appendix of \cite{gouere}). We denote the $\ell^1$ norm on $\mathbb{R}^2$ by $\|\cdot\|_1$ and the $\ell^2$ norm by $\| \cdot \|_2 .$ Since the limit shape is bounded and has nonempty interior, there are constants $0 < C_1, C_2 < \infty$ such that
\begin{equation}\label{eq: normequivalence}
C_1 \|x\|_2 \leq g(x) \leq C_2 \|x\|_2 \text{ for all } x \in \mathbb{R}^2\ .
\end{equation}

We recall the fact that under {\bf A1} or {\bf A2},
\begin{equation}\label{eq: finitesecondmoment}
\mathbb{E} \tau(x,y)^2 < \infty \text{ for all } x,y \in \mathbb{R}^2\ .
\end{equation}
This was proved in \cite[Lemma~3.1]{coxdurrett} assuming {\bf A1} and in the other case it follows directly from the fact that $\mathbb{E} \omega_e^{2+\e}<\infty$ for some $\e>0$.

We write $x \cdot y$ for the standard dot product between $x$ and $y$ in $\mathbb{R}^2$. 

\begin{center}
{\bf For the rest of the paper we assume A1 or A2.}
\end{center}

\subsection{Structure of the paper}
In the next section, we give basic properties of Busemann functions and geodesic graphs. In Section~\ref{sec: BID} we introduce Busemann increment configurations and construct probability measures on them. Next we reconstruct Busemann functions, and in Section~\ref{sec: limits} we prove a shape theorem for the reconstruction.  Section~\ref{sec: GG} begins the study of distributional limits $\mathbb{G}$ of geodesic graphs, where we show that all paths are asymptotically directed in a sector given by the reconstructed Busemann function. In Section~\ref{sec: coalesceG} we show coalescence of all paths in $\mathbb{G}$. We use all of these tools in Section~\ref{sec: proofs} to prove the main results of the paper.

\section{Busemann functions and geodesic graphs}

In this section we will give basic properties of Busemann functions and geodesic graphs. These will be carried over through weak limits to a space introduced in the next section.

\subsection{Busemann functions}

For any $S \subseteq \mathbb{R}^2$ and configuration $\omega$, we define the Busemann function $B_S : \Z2\times \Z2 \to \mathbb{R}$ as
\[
B_S(x,y) = \tau(x,S) - \tau(y,S)\ ,
\]
This function measures the discrepancy between travel times from $x$ and $y$ to $S$. We list below some basic properties of Busemann functions. One of the most interesting is the additivity property 1. It is the reason that the asymptotic shape for the Busemann function is a half space whereas the asymptotic shape for $\tau$ is a compact set.

\begin{prop}\label{prop: busemannprop1}
Let $S \subseteq \mathbb{R}^2$. The Busemann function $B_S$ satisfies the following properties $\mathbb{P}$-almost surely for $x,y,z \in \Z2$:
\begin{enumerate}
\item (Additivity)
\begin{equation}\label{eq: busemannadditivity}
B_S(x,y) = B_S(x,z) + B_S(z,y)\ .
\end{equation}
\item for $i=1,2$,
\begin{equation}\label{eq: busemanntranscov}
B_S(x,y)(T_{\unitv_i} \omega) = B_{S-\unitv_i}(x-\unitv_i,y-\unitv_i)(\omega)\ .
\end{equation}
Therefore the finite-dimensional distributions of $B_S$ obey a translation invariance:
\[
\left( B_S(x,y) \right) \underset{d}{=} \left(B_{S-\unitv_i}(x-\unitv_i,y-\unitv_i) \right)\ .
\]
\item
\begin{equation}\label{eq: bboundtau}
|B_S(x,y)| \leq \tau(x,y)\ .
\end{equation}
\end{enumerate}
\end{prop}
\begin{proof}
The first property follows from the definition. The third is a consequence of subadditivity \eqref{eq: subadditivity} of $\tau(y,S)$. The second item follows from the statement \eqref{eq: transcov} for passage times.
\end{proof}

The last property we need regards the relation between geodesics and Busemann functions. Though it is simple, it will prove to be important later.
\begin{prop}\label{prop: busemannpassagetime}
Let $S \subseteq \mathbb{R}^2$ and $x \in \Z2$. If $\gamma$ is a geodesic from $x$ to $S$ and $y$ is a vertex of $\gamma$ then $B_S(x,y) = \tau(x,y)$.
\end{prop}

\begin{proof}
Write $\tau_\gamma(x,y)$ for the passage time along $\gamma$ between $x$ and $y$. Since every segment of a geodesic itself a geodesic, $\tau(x,S)-\tau(y,S) = \tau_\gamma(x,S) - \tau_\gamma(y,S) = \tau_\gamma(x,y) = \tau(x,y)$.
\end{proof}

Using this proposition and additivity of the Busemann function we can relate $B_S(x,y)$ to coalescence. If $\gamma_x$ and $\gamma_y$ are geodesics from $x$ and $y$ to $S$ (respectively) and they meet at a vertex $z$ then $B_S(x,y) = \tau(x,z)-\tau(y,z)$. This is a main reason why Busemann functions are useful for studying coalescence of geodesics.

\subsection{Geodesic graphs}\label{subsec: GG}

For any $S \subseteq \Z2$ and configuration $\omega$, we denote the set of edges in all geodesics from a point $v \in \Z2$ to $S$ as $G_S(v)$. We regard each geodesic in $G_S(v)$ as a directed path, giving orientation $\langle x,y \rangle$ to an edge if $\tau(x,S) \geq \tau(y,S)$ (the direction in which the edge is crossed), and set $\vec{G}_S(v)$ to be the union of these directed edges. Let $\mathbb{G}_S(\omega)$ be the directed graph induced by the edges in $\cup_v \vec{G}_S(v)$. Last, define the configuration $\eta_S(\omega)$ of directed edges by
\[
\eta_S(\omega)(\langle x,y \rangle) = \begin{cases}
1 & \text{if }\langle x,y \rangle \in \vec{G}_S(v) \text{ for some } v \\
0 & \text{otherwise}
\end{cases}\ .
\]
For $S \subseteq \mathbb{R}^2$ we define $\eta_S(\omega)$ and $\mathbb{G}_S(\omega)$ using $\hat S$ as in \eqref{eq: hatS}.

\begin{prop}\label{prop: firstGG}
Let $S \subseteq \R2$. The graph $\mathbb{G}_S$ and the collection $(\eta_S)$ satisfy the following properties $\mathbb{P}$-almost surely.
\begin{enumerate}
\item Every finite directed path is a geodesic. It is a subpath of a geodesic ending in $S$.
\item If there is a directed path from $x$ to $y$ in $\mathbb{G}_S$ then $B_S(x,y) = \tau(x,y)$.
\item For $i=1,2$,
\begin{equation}\label{eq: GGtranscov}
\eta_S(e)(T_{\unitv_i}\omega) = \eta_{S-\unitv_i}(e-\unitv_i)(\omega)\ .
\end{equation}
Therefore the finite dimensional distributions of $\eta_S$ obey a translation invariance:
\[
(\eta_S(e)) \underset{d}{=} (\eta_{S-\unitv_i}(e-\unitv_i))\ .
\]
\end{enumerate}
\end{prop}
\begin{proof}
The third property follows from translation covariance of passage times \eqref{eq: transcov}. The second property follows from the first and Proposition~\ref{prop: busemannpassagetime}.

To prove the first, let $\gamma$ be a directed path in $\mathbb{G}_S$ and write the edges of $\gamma$ in order as $e_1,\ldots, e_n$. Write $J \subseteq \{1, \ldots, n\}$ for the set of $k$ such that the path $\gamma_k$ induced by $e_1, \ldots, e_k$ is a subpath of a geodesic from some vertex to $S$. We will show that $n \in J$. By construction of $\mathbb{G}_S$, the edge $e_1$ is in a geodesic from some point to $S$, so $1 \in J$. Now suppose that $k \in J$ for some $k < n$; we will show that $k+1 \in J$. Take $\sigma$ to be a geodesic from a point $z$ to $S$ which contains $\gamma_k$ as a subpath. Write $\sigma'$ for the portion of the path from $z$ to the far endpoint $v_k$ of $e_k$ (the vertex to which $e_k$ points). The edge $e_{k+1}$ is also in $\mathbb{G}_S$ so it is in a geodesic from some point to $S$. If we write $\hat \sigma$ for the piece of this geodesic from $v_k$ of $e_k$ to $S$, we claim that the concatenation of $\sigma'$ with $\hat \sigma$ is a geodesic from $z$ to $S$. To see this, write $\tau_{\tilde \gamma}$ for the passage time along a path $\tilde \gamma$:
\[
\tau(z,S) = \tau_\sigma(z,v_k) + \tau_\sigma(v_k,S) = \tau_{\sigma'}(z,v_k) + \tau_{\hat \sigma}(v_k,S)\ .
\]
The last equality holds since both the segment of $\hat \sigma$ from $v_k$ to $S$ and the segment of $\sigma$ from $v_k$ to $S$ are geodesics, so they have equal passage time. Hence $k+1 \in J$ and we are done.
\end{proof}

Note that each vertex $x \notin \hat S$ has out-degree at least 1 in $\mathbb{G}_S$. Furthermore it is possible to argue using part 1 of the previous proposition and the shape theorem that there are no infinite directed paths in $\mathbb{G}_S$. Since we will not use this result later, we omit the proof. Once we take limits of measures on such graphs later, infinite paths will appear.

If $\mathbb{P}$ has unique passage times, we can say more about the structure of $\mathbb{G}_S$. 
\begin{prop}\label{prop: secondGG}
Assume {\bf A1'} or {\bf A2'}. The following properties hold $\mathbb{P}$-almost surely.
\begin{enumerate}
\item Each vertex $x \notin \hat S$ has out-degree 1. Here $\hat S$ is defined as in \eqref{eq: hatS}.
\item Viewed as an undirected graph, $\mathbb{G}_S$ has no circuits.
\end{enumerate}
\end{prop}
\begin{proof}
For the first property note that every vertex $x\notin \hat S$ has out-degree at least 1 because there is a geodesic from the vertex to $S$ and the first edge is directed away from $x$. Assuming $x$ has out-degree at least 2 then we write $e_1$ and $e_2$ for two such directed edges. By the previous proposition, there are two geodesics $\gamma_1$ and $\gamma_2$ from $x$ to $S$ such that $e_i \in \gamma_i$ for $i=1,2$. If either of these paths returned to $x$ then there would exist a finite path with passage time equal to 0. By the ergodic theorem there would then be infinitely many distinct paths with passage time 0 (with positive probability), contradicting unique passage times. This implies that $\gamma_1$ and $\gamma_2$ have distinct edge sets. However, they have the same passage time, again contradicting unique passage times.

For the second property suppose that there is a circuit in the undirected version of $\mathbb{G}_S$. Each vertex has out-degree 1, so this is actually a directed circuit and thus a geodesic. But then it has passage time zero, giving a contradiction as above.
\end{proof}

Property 2 implies that $\mathbb{G}_S$, viewed as an undirected graph, is a forest. It has more than one component if and only if $\hat S$ has size at least 2. We will see later that after taking limits of measures on these graphs, the number of components will reduce to 1.

\section{Busemann increment distributions}\label{sec: BID}

We are interested in taking limits of measures on Busemann functions and geodesic graphs. 
We will choose a one-parameter family of lines $L_\alpha = L + \alpha \mathbf{v}$ for $\mathbf{v}$ a normal vector to $L$ and consider the Busemann functions $B_{L_\alpha}(x,y)$. The main question is whether or not the limit
\begin{equation}\label{eq: newmanlimit}
\lim_{\alpha \to \infty} B_{L_\alpha}(x,y)
\end{equation}
exists for $x,y \in \mathbb{Z}^2$. If one could show this, then one could prove many results about FPP, for instance, that infinite geodesics with an asymptotic direction always exist. Under an assumption of uniformly positive curvature of the limit shape $\mathcal{B}$ and exponential moments for the common distribution of the $\omega_e$'s (in the case that $\mathbb{P}$ is a product measure) Newman \cite{Newman95} has shown the existence of this limit for Lebesgue-almost every unit vector $\mathbf{v}$.

We will try to overcome the difficulty of existence of limits \eqref{eq: newmanlimit} by enlarging the space to work with subsequential limits in a systematic way. This technique is inspired by work \cite{AD, ADNS} on ground states of short-range spin glasses.

\subsection{Definition of $\mu$}\label{sec: mudef}

We begin by assigning a space for our passage times. Let $\Omega_1 = \mathbb{R}^{\mathbb{Z}^2}$ be a copy of $\Omega$. A sample point in $\Omega_1$ we call $\omega$ as before. Our goal is to enhance this space to keep track of Busemann functions and geodesic graphs. We will take limits in a fixed direction, so for the remainder of this section, let $\varpi\in \partial \mathcal{B}$ and let $g_\varpi$ be any linear functional on $\mathbb{R}^2$ that takes its maximum on $\mathcal{B}$ at $\varpi$ with $g_\varpi(\varpi)=1$. The nullspace of $g_\varpi$ is then a translate of a supporting line for $\mathcal{B}$ at $\varpi$.
For $\alpha \in \mathbb{R}$, define 
\[
L_\alpha = \left\{ x \in \mathbb{R}^2: g_\varpi(x) = \alpha \right\}\ .
\]
For future reference, we note the inequality
\begin{equation}\label{eq: gvarpibound}
\text{for all } x \in \mathbb{R}^2,~ g_\varpi(x) \leq g(x)\ .
\end{equation}
It clearly holds if $x \neq 0$. Otherwise since $x/g(x) \in \mathcal{B}$, $1 \geq g_\varpi(x/g(x)) = g_\varpi(x)/g(x)$.

Given $\alpha \in \mathbb{R}$ and $\omega \in \Omega_1$, write $B_\alpha(x,y)(\omega) = B_{L_\alpha}(x,y)(\omega)$. Define the space $\Omega_2 = (\mathbb{R}^2)^{\mathbb{Z}^2}$ with the product topology and Borel sigma-algebra and the {\it Busemann increment configuration} $B_\alpha(\omega) \in \Omega_2$ as
\begin{align*}
B_\alpha(\omega) = \big( \, B_\alpha(v, v+\unitv_1),\, B_\alpha(v,v+\unitv_2)\, \big)_{v \in \Z2}\ .
\end{align*}

We also consider directed graphs of geodesics. These are points in a directed graph space $\Omega_3 = \{0,1\}^{\vec{\mathcal{E}}^2}$, where $\vec{\mathcal{E}}^2$ is the set of oriented edges $\langle x,y \rangle$ of $\mathbb{Z}^2$, and we use the product topology and Borel sigma-algebra. For $\eta \in \Omega_3$, write $\mathbb{G} = \mathbb{G}(\eta)$ for the directed graph induced by the edges $e$ such that $\eta(e) = 1$. Using the definition from the last section, set
\[
\eta_\alpha(\omega) = \eta_{L_\alpha}(\omega) \in \Omega_3 \text{ and } \mathbb{G}_\alpha(\omega) = \mathbb{G}(\eta_\alpha(\omega)) \text{ for } \alpha \in \mathbb{R}\ .
\]

Set $\widetilde \Omega = \Omega_1 \times \Omega_2 \times \Omega_3$, equipped with the product topology and Borel sigma-algebra; 
\[
(\omega, \Theta,\eta) = (\omega(e), \theta_1(x),\theta_2(x),\eta(f) : e \in \mathcal{E}^2, x \in \Z2,~f \in \vec{\mathcal{E}}^2)
\] 
denotes a generic element of the space $\widetilde \Omega.$ Define the map
\begin{equation}\label{eq: phidef}
\Phi_\alpha: \Omega_1 \longrightarrow \widetilde \Omega  \text{ by } \omega \mapsto (\omega, B_\alpha(\omega),\eta_\alpha(\omega))\ .
\end{equation}
Because $\Phi_\alpha$ is measurable, we can use it to push forward the distribution $\passage$ to a probability measure $\mu_\alpha$ on $\widetilde \Omega$. Given the family $(\mu_\alpha)$ and $n \in \mathbb{N}$, we define the empirical average
\begin{equation}\label{eq: munstar}
\mu_n^*\left( \cdot \right) := \frac{1}{n} \int_{0}^n \mu_\alpha \left( \cdot \right) \md \alpha. 
\end{equation}
To prove that this defines a probability measure, one must show that for each measurable $A \subseteq \widetilde \Omega$, the map $\alpha \mapsto \mu_\alpha(A)$ is Lebesgue-measurable. The proof is deferred to Appendix~\ref{sec: appendix}.

From $B_\alpha(x,y) \leq \tau(x,y)$, the sequence $\left( \mu_n^* \right)_{n=1}^{\infty}$ is seen to be tight and thus has a subsequential weak limit $\mu.$ We will call the marginal of $\mu$ on $\Omega_2$ a {\it Busemann increment distribution} and the marginal on $\Omega_3$ a {\it geodesic graph distribution}. It will be important to recall the Portmanteau theorem, a basic result about weak convergence. The following are equivalent if $(\nu_k)$ is a sequence of Borel probability measures on a metric space $X$:
\begin{align}
\lim_{k \to \infty} \nu_k &\to \nu \text{ weakly } \nonumber \\
\limsup_{k \to \infty} \nu_k(A) &\leq \nu(A) \text{ if } A \text{ is closed} \label{eq: kallenbergclosed}\\
\liminf_{k \to \infty} \nu_k(A) &\geq \nu(A) \text{ if } A \text{ is open} \label{eq: kallenbergopen}\ .
\end{align}
(See, for example, \cite[Theorem~3.25]{Kallenberg}.) Because $\widetilde \Omega$ is metrizable, these statements apply.

In this section and the next, we prove general properties about the measure $\mu$ and focus on the marginal on $\Omega_2$. In Sections~\ref{sec: GG} and \ref{sec: coalesceG} we study the marginal on $\Omega_3$ and in Section~\ref{sec: proofs} relate results back to the original FPP model. It is important to remember that $\mu$ depends among other things not only on $\varpi$, but on the choice of the linear functional $g_\varpi$. We will suppress mention of $\varpi$ in the notation. Furthermore we will use $\mu$ to represent the measure and also its marginals. For instance, if we write $\mu(A)$ for an event $A \subseteq \Omega_2$ we mean $\mu(\Omega_1 \times A \times \Omega_3)$.

\subsection{Translation invariance of $\mu$.}
We will show that $\mu$ inherits translation invariance from \passage. The natural translations $\tilde{T}_m,~m=1,2$ act on $\widetilde \Omega$ as follows:
\[
\left[ \tilde{T}_m (\omega,\Theta,\eta) \right](e,x,f) = \left( \omega_{e-\unitv_m}, \theta_1 (x - \unitv_m), \theta_2 (x - \unitv_m), \eta(f-\unitv_m) \right)\ .
\]
Here, for example, we interpret $e-\unitv_m$ for the edge $e= (y,z)$ as $(y-\unitv_m,z-\unitv_m)$.

\begin{lem}
\label{translate_1}
For any $\alpha \in \mathbb{R}$ and $m=1,2$, $\mu_\alpha \circ \tilde T_m = \mu_{\alpha + g_{\varpi}(\unitv_m)}$.

\begin{proof}
Let $A$ be a cylinder event for the space $\widetilde \Omega$ of the form
\[
A = \left\{ \omega_{e_i} \in \mathbf{B}_i, \theta_{r_j}(x_j) \in \mathbf{C}_j, \eta(f_k) = a_k : i = 1, \ldots, l, ~j = 1, \ldots, m, ~k=1, \ldots, n \right\}\ ,
\]
where each $\mathbf{B}_i, \mathbf{C}_j$ is a (real) Borel set with $a_k \in \{0,1\}$, each $r_j \in \{1,2\}$, and each $e_i \in \mathcal{E}^2, x_j \in \mathbb{Z}^2$ and $f_k \in \vec{\mathcal{E}}^2$.  We will show that for $m=1,2$,
\begin{equation}\label{eq: chazisagooddog}
\mu_\alpha \left(\tilde{T}_m^{-1} A\right) = \mu_{\alpha + g_{\varpi}(\unitv_m)}(A)\ .
\end{equation}
Such $A$ generate the sigma-algebra so this will imply the lemma. For $m \in \{1,2\}$,
\begin{equation}
\label{cylindershift}
\tilde{T}_m^{-1}(A) = \left\{\omega_{e_i - \unitv_m} \in \mathbf{B}_i, \theta_{r_j}(x_j-\unitv_m) \in \mathbf{C}_j, \eta(f_k-\unitv_m) = a_k \right\}\ . \nonumber
\end{equation}
Rewriting $\mu_\alpha(\cdot) = \mathbb{P}(\Phi_\alpha^{-1}(\cdot))$ and using the definition of $\Phi_\alpha$ \eqref{eq: phidef}, 
\[
\mu_\alpha(\tilde T_m^{-1}(A)) = \passage \left( \omega_{e_i-\unitv_m} \in \mathbf{B}_i, B_\alpha(x_j - \unitv_m, x_j - \unitv_m + \unitv_{r_j} ) \in \mathbf{C}_j, \eta_\alpha(f_k-\unitv_m)(\omega) = a_k \right)\ .
\]
Note that translation invariance of $\passage$ allows to shift the translation by $\unitv_m$ from the arguments of $\omega$, $B_\alpha$ and $\eta_\alpha$ to the position of the line $L_\alpha$.  We have equality in distribution:
\[
\omega_{e-\unitv_m} \underset{d}{=} \omega_e,~ B_\alpha(x - \unitv_m, y - \unitv_m) \underset{d}{=} B_\beta(x,y) \text{ and } \eta_\alpha(e-\unitv_m) \underset{d}{=} \eta_\beta(e)\ ,
\]
where $\beta=\alpha + g_\varpi (\unitv_m)$. In fact, using the translation covariance statements \eqref{eq: transcov}, \eqref{eq: busemanntranscov} and \eqref{eq: GGtranscov}, equality of the above sort holds for the joint distribution of the $\omega$'s, Busemann increments and graph variables appearing in the event $A.$ This proves \eqref{eq: chazisagooddog}.
\end{proof}
\end{lem}

\begin{prop}
	$\mu$ is invariant under the translations $\tilde{T}_m$, $m=1,2$.
\end{prop}
\begin{proof}
	Let $f$ be a continuous function (bounded by $D \geq 0$) on the space $\widetilde \Omega,$ and fix $\epsilon > 0.$  Choose an increasing sequence $(n_k)$ such that $\mu_{n_k}^* \to \mu$ weakly as $k \to \infty$. We can then find $k_0$ such that
$|\mu(f) - \mu_{n_k}^*(f)| < \epsilon/3$ for $k > k_0.$  
By Lemma~\ref{translate_1}, $\mu_\alpha \circ \tilde T_m = \mu_{\alpha + g_{\varpi}(\unitv_m)}$ for $m=1,2$. Therefore 
\begin{align*}
	\left[\mu_{n_k}^* \circ \tilde{T}_m\right] \left(f\right) &= \frac{1}{n_k}\int_{g_{\varpi}(\unitv_m)}^{n_k+g_{\varpi}(\unitv_m)} \mu_\alpha \left(f \right)  \md \alpha\\
	\Rightarrow \left| \left[\mu_{n_k}^* \circ \tilde{T}_m\right] \left(f\right) - \mu_{n_k}^* \left(f\right) \right|  &\leq
		\frac{1}{n_k}\left| \int_{0}^{g_{\varpi}(\unitv_m)} \mu_\alpha \left(f \right)  \md \alpha \right| +  \frac{1}{n_k} \left| \int_{n_k}^{n_k+g_{\varpi}(\unitv_m)} \mu_\alpha \left(f \right)  \md \alpha \right|\\
	&\leq \frac{2 g_{\varpi}(\unitv_m)D}{n_k} \rightarrow 0 \text{ as } k \to \infty\ .
\end{align*}
As $\tilde T_m$ is a continuous on $\widetilde \Omega$, $(\mu_{n_k}^* \circ \tilde{T}_m) $ converges weakly to $\mu \circ \tilde{T}_m,$ so there exists $k_1>k_0$ such that 
$|\mu\circ \tilde T_m(f) - \mu_{n_k}^* \circ \tilde T_m (f)| < \epsilon/3$ for all $k > k_1,$ and $k_2>k_1$ with $2g_{\varpi}(\unitv_m)D/n_{k_2} < \epsilon/3.$  So $|\mu(f) - \mu \circ \tilde{T}_m (f)|  < \epsilon \text{ for all } \e>0$, giving $\mu = \mu \circ \tilde T_m$.
\end{proof}

\subsection{Reconstructed Busemann functions}
We wish to reconstruct an ``asymptotic Busemann function" $f: \Z2 \rightarrow \mathbb{R}$ by summing the Busemann increments of $\Theta \in \Omega_2$. That $\Theta$ is almost surely curl-free allows the construction to proceed independent of the path we sum over. For this we need some definitions.

Given $\Theta \in \Omega_2$, $x \in \mathbb{Z}^2$ and $z \in \mathbb{Z}^2$ with $\|z\|_1=1$ we set $\theta(x,z) = \theta(x,z)(\Theta)$ equal to
\[
\theta(x,z) = \begin{cases}
\theta_1(x) & z = \unitv_1 \\
\theta_2(x) & z= \unitv_2 \\
-\theta_1(x-\unitv_1) & z=-\unitv_1 \\
-\theta_2(x-\unitv_2) & z = -\unitv_2
\end{cases}\ .
\]
For any finite lattice path $\gamma$ we write its vertices in order as $x_1, \ldots, x_n$ and set
\[
f(\gamma) = f(\gamma)(\Theta) = \sum_{i=1}^{n-1} \theta(x_i, x_{i+1}-x_i)\ .
\]

\begin{lem}\label{lem: curlfree}
With $\mu$-probability one, $f$ vanishes on all circuits:
\[
\mu \left( f(\gamma)=0 \text{ for all circuits } \gamma \right) = 1\ .
\]
\end{lem}

\begin{proof}
Pick a circuit $\gamma$ and let $A\subseteq \widetilde \Omega_2$ denote the event $\{ \Theta : f(\gamma) = 0\}$. Choose an increasing sequence $(n_k)$ such that $\mu_{n_k}^* \to \mu$ weakly. For fixed $\gamma$, $f(\gamma)$ is a continuous function on $\widetilde \Omega$, so the event $A$ is closed, giving $\mu(A) \geq \limsup_{k} \mu_{n_k}^*(A)$ by \eqref{eq: kallenbergclosed} . However, for each $\alpha$, by additivity of $B_\alpha(\cdot,\cdot)$ (see \eqref{eq: busemannadditivity}),
\[
\mu_\alpha(A) = \mathbb{P}\left(\sum_{i=1}^n B_\alpha(x_i,x_{i+1}) = 0\right) = 1\ .
\]
Thus $\mu_n^*(A) = 1$ for all $n$ and $\mu(A)=1$. There are countably many $\gamma$'s so we are done.
\end{proof}

Using the lemma we may define the reconstructed Busemann function. Fix a deterministic family of finite paths $\{\gamma_{x,y}\}$, one for each pair $(x,y) \in \mathbb{Z}^2$ and define 
\[
f(x,y) = f(x,y)(\Theta) := f(\gamma_{x,y})\ .
\] 
Although we use fixed paths $\gamma_{x,y}$, this is only to ensure that $f$ is a continuous function on $\widetilde \Omega$. Actually, for any $\Theta$ in the $\mu$-probability one set of Lemma~\ref{lem: curlfree} and vertices $x,y \in \mathbb{Z}^2$ we could equivalently define $f(x,y) = f(\gamma)$, where $\gamma$ is any finite lattice path from $x$ to $y$. To see that it would then be well-defined (that is, only a function of $x,y$ and the configuration $\Theta$) is a standard argument. If we suppose that $\gamma_1$ and $\gamma_2$ are finite lattice paths from $x$ to $y$ and $\Theta$ is given as above, the concatenation of $\gamma_1$ with $\gamma_2$ (traversed in the opposite direction) is a circuit and thus has $f$-value zero. However, by definition, this is the difference of $f(\gamma_1)$ and $f(\gamma_2)$ and proves the claim.


We now give some properties about asymptotic Busemann functions that come over from the original model. The third says that $f$ retains translation covariance. This will allow us to prove the existence of almost-sure limits using the ergodic theorem in the next section.
\begin{prop}
\label{bd_f_by_tau_prop}
The reconstructed Busemann function satisfies the following properties for $x,y,z \in \Z2$.
\begin{enumerate}
\item 
\begin{equation}\label{eq: additivity}
f(x,y) + f(y,z) = f(x,z) ~\mu\text{-almost surely}\ . 
\end{equation}
\item For $m=1,2$ 
\begin{equation}\label{eq: ftranslation}
f(x,y)(\tilde T_m \Theta) = f(x-\unitv_m,y-\unitv_m)(\Theta) ~\mu\text{-almost surely}\ .
\end{equation}
\item \begin{equation}\label{eq: continuity}
f(x,y):\widetilde \Omega \to \mathbb{R} \text{ is continuous}\ .
\end{equation}
\item $f$ is bounded by $\tau$:
\begin{equation}\label{eq: fboundtau}
|f(x,y)| \leq \tau(x,y)) ~\mu\text{-almost surely}\ .
\end{equation}
\end{enumerate}
\end{prop}

\begin{proof}
The first two properties follow from path-independence of $f$ and the third holds because $f$ is a sum of finitely many Busemann increments, each of which is a continuous function. We show the fourth property. For $x,y \in \Z2$, the event
\[
\{(\omega,\Theta) : |f(x,y)(\Theta)| - \tau(x,y)(\omega) \leq 0\}
\]
is closed because $|f(x,y)|-\tau(x,y)$ is continuous. For every $\alpha$, \eqref{eq: bboundtau} gives $|B_\alpha(x,y)| \leq \tau(x,y)$ with $\mathbb{P}$-probability one, so the above event has $\mu_\alpha$-probability one. Taking limits and using \eqref{eq: kallenbergclosed}, $\mu(|f(x,y)(\Theta)| \leq \tau(x,y)(\omega)) = 1$.
\end{proof}

\subsection{Expected value of $f$}
In this section we compute $\mathbb{E}_\mu f(0,x)$ for all $x \in \mathbb{Z}^2$. The core of our proof is a argument from Hoffman \cite{Hoffman2}, which was developed using an averaging argument due to Garet-Marchand \cite{GM}. The presentation we give below is inspired by that of Gou\'er\'e \cite[Lemma~2.6]{gouere}. In fact, the proof shows a stronger statement. Without need for a subsequence,
\[
\mathbb{E}_{\mu_{n}^*}f(0,x) \to g_\varpi(x)\ .
\]

\begin{thm}\label{thm: expected_value}
For each $x \in \mathbb{Z}^2$, $\mathbb{E}_\mu f(0,x) = g_\varpi(x)$.
\end{thm}

\begin{proof}
We will use an elementary lemma that follows from the shape theorem.
\begin{lem}\label{lem: pointtoplane}
The following convergence takes place almost surely and in $L^1(\mathbb{P})$:
\[
\frac{\tau(0,L_\alpha)}{\alpha} \to 1 \text{ as } \alpha \to \infty\ .
\]
\end{lem}

\begin{proof}
Since $\alpha \varpi \in L_\alpha$,
\[
\limsup_{\alpha \to \infty} \frac{\tau(0,L_\alpha)}{\alpha} \leq \lim_{\alpha \to \infty} \frac{\tau(0,\alpha \varpi)}{\alpha} = 1\ .
\]
On the other hand, given $\e>0$ and any $\omega$ for which the shape theorem holds, we can find $K$ such that for all $x \in \mathbb{R}^2$ with $\|x\|_1 \geq K$, $\tau(0,x) \geq g(x)(1-\e)$. So if $\alpha$ is large enough that all $x \in L_\alpha$ have $\|x\|_1 \geq K$, then we can use \eqref{eq: gvarpibound}:
\[
\tau(0,L_\alpha) = \min_{x \in L_\alpha} \tau(0,x) \geq (1-\e) \min_{x \in L_\alpha} g(x) \geq (1-\e)\alpha\ .
\]
Consequently, $\liminf_{\alpha \to \infty} \tau(0,L_\alpha)/\alpha \geq 1$, giving almost sure convergence in the lemma.

For $L^1$ convergence, note $0 \leq \tau(0,L_\alpha)/\alpha \leq \tau(0,\alpha \varpi)/\alpha$, so the dominated convergence theorem and $L^1$ convergence of point to point passage times completes the proof.
\end{proof}

For any $x \in \mathbb{Z}^2$ and integer $n \geq 1$, use the definition of $\mu_n^*$ to write 
\[
\mathbb{E}_{\mu_n^*}(f(-x,0)) = \frac{1}{n} \left[ \int_0^n \mathbb{E} \tau(-x,L_\alpha)~\md \alpha - \int_0^n \mathbb{E} \tau(0, L_\alpha) ~\md \alpha \right]\ .
\]
Using translation covariance of passage times,
\[
\int_0^n \mathbb{E} \tau(-x,L_\alpha)~\md \alpha = \int_0^n \mathbb{E} \tau(0,L_{\alpha + g_\varpi(x)})~\md \alpha = \int_{g_\varpi(x)}^{n+g_\varpi(x)} \mathbb{E} \tau(0,L_\alpha) ~\md \alpha\ .
\]
Therefore
\begin{equation}\label{eq: almostformula}
\mathbb{E}_{\mu_n^*}(f(-x,0)) = \frac{1}{n} \left[ \int_{n}^{n+g_\varpi(x)} \mathbb{E} \tau(0,L_\alpha) ~\md \alpha - \int_0^{g_\varpi(x)} \mathbb{E} \tau(0,L_\alpha) ~\md \alpha \right] \ .
\end{equation}

Choose $(n_k)$ to be an increasing sequence such that $\mu_{n_k}^* \to \mu$ weakly. We claim that
\begin{equation}\label{eq: momentconvergence}
\mathbb{E}_{\mu_{n_k}^*} f(-x,0) \to \mathbb{E}_\mu f(-x,0)\ .
\end{equation}
To prove this, note that for any $R>0$, if we define the truncated variable
\[
f_R(-x,0) = \text{ sgn} f(-x,0) \min\{R, |f(-x,0)|\}\ ,
\]
then continuity of $f$ on $\widetilde \Omega$ gives $\mathbb{E}_{\mu_{n_k}^*} f_R(-x,0) \to \mathbb{E}_\mu f_R(-x,0)$. To extend this to \eqref{eq: momentconvergence}, it suffices to prove that for each $\e>0$, there exists $R>0$ such that
\begin{equation}\label{eq: limsupcondition}
\limsup_{k \to \infty} \mathbb{E}_{\mu_{n_k}^*} |f(-x,0)| I(|f(-x,0)| \geq R) < \e\ ,
\end{equation}
where $I(A)$ is the indicator of the event $A$. Because $\mathbb{E}_{\mu_{n_k}^*} f(-x,0)^2 \leq \mathbb{E} \tau(-x,0)^2 < \infty$ for all $k$ by \eqref{eq: finitesecondmoment}, condition \eqref{eq: limsupcondition} follows from the Cauchy-Schwarz inequality. This proves \eqref{eq: momentconvergence}. 

Combining \eqref{eq: almostformula} and \eqref{eq: momentconvergence}, we obtain the formula
\begin{equation}\label{eq: almostdoneagain}
\mathbb{E}_\mu f(-x,0) = \lim_{k \to \infty} \frac{1}{n_k} \int_{n_k}^{n_k+g_\varpi(x)} \mathbb{E} \tau(0,L_\alpha) ~\md \alpha = \lim_{k \to \infty} \int_0^{g_\varpi(x)} \frac{\mathbb{E} \tau(0,L_{\alpha + n_k})}{n_k} ~\md \alpha\ .
\end{equation}
By Lemma~\ref{lem: pointtoplane}, for each $\alpha$ between $0$ and $g_\varpi(x)$,
\[
\lim_{k \to \infty} \frac{\mathbb{E} \tau(0,L_{\alpha + n_k})}{n_k} = \lim_{k \to \infty} \frac{\mathbb{E} \tau(0,L_{\alpha + n_k})}{\alpha + n_k} \cdot \frac{\alpha + n_k}{n_k} = 1\ .
\]
So using $\mathbb{E} \tau(0,L_{\alpha + n_k}) \leq \mathbb{E} \tau(0,L_{2n_k})$ for large $k$, we can pass the limit under the integral in \eqref{eq: almostdoneagain} to get $\mathbb{E}_\mu f(0,x) = \mathbb{E}_\mu f(-x,0) = g_\varpi(x)$.
\end{proof}

\section{Limits for reconstructed Busemann functions}\label{sec: limits}

In this section we study the asymptotic behavior of the reconstructed Busemann function $f$. We will see that $f$ is asymptotically a projection onto a line and if the boundary of the limit shape is differentiable at $\varpi$, we give the explicit form of the hyperplane. Without this assumption we show that the line is a translate of a supporting line for $\mathcal{B}$ at $\varpi$.

One of the advantages of constructing $f$ from our measure $\mu$ is that we can use the ergodic theorem and translation invariance to show the existence of limits. This gives us almost as much control on the Busemann function as we would have if we could show existence of the limit in \eqref{eq: newmanlimit}. If we knew this, we would not need differentiability at $\varpi$ to deduce the form of the random hyperplane for $f$; we could derive it from ergodicity and symmetry.

\subsection{Radial limits}

In this section we will prove the existence of radial limits for $f$. This is the first step to deduce a shape theorem, which we will do in the next section. We extend the definition of $f$ to all of $\mathbb{R}^2 \times \mathbb{R}^2$ in the usual way: $f(x,y)$ is defined as $f(\tilde x, \tilde y)$ where $\tilde x$ and $\tilde y$ are the unique points in $\Z2$ such that $x \in \tilde x + [-1/2,1/2)^2$ and $y \in \tilde y + [-1/2,1/2)^2$.

\begin{prop}\label{prop: rationallimits}
Let $q \in \mathbb{Q}^2$. Then
\[
\rho_q := \lim_{n \to \infty} \frac{1}{n} f(0,nq) \text{ exists } \mu\text{-almost surely}\ .
\]
\end{prop}

\begin{proof}
Choose $M \in \mathbb{N}$ such that $Mq \in \mathbb{Z}^2$. We will first show that
\begin{equation}\label{eq: firstrationallimit}
\lim_{n \to \infty} \frac{1}{Mn} f(0,nMq) \text{ exists } \mu\text{-almost surely}\ .
\end{equation}
To do this, we note that since $\tau(0,Mq) \in L^2(\mu)$ (from \eqref{eq: finitesecondmoment}), it is also in $L^1$. Using \eqref{eq: fboundtau}, $f(0,Mq) \in L^1(\mu)$ as well. Define the map $\tilde T_q$ on $\Omega_2$ as
\[
\left[ \tilde T_q \Theta \right] (x) = (\theta_1(x-Mq),\theta_2(x-Mq))\ .
\]
This is a composition of maps $\tilde T_m$, $m=1,2$, so it is measure-preserving. By \eqref{eq: additivity} and \eqref{eq: ftranslation},
\[
f(0,nMq)(\Theta) = \sum_{i=1}^n f((i-1)Mq,iMq)(\Theta) = \sum_{i=0}^{n-1} f(0,Mq)(\tilde T_q^{-i} (\Theta))\ .
\]
Applying the ergodic theorem finishes the proof of \eqref{eq: firstrationallimit}.

To transform \eqref{eq: firstrationallimit} into the statement of the proposition we need to ``fill in the gaps.'' Choose $M$ as above and for any $n$ pick $a_n \in \mathbb{Z}$ such that $a_nM \leq n < (a_n+1)M$. Then
\[
\left| \frac{f(0,nq)}{n} - \frac{f(0,a_nMq)}{a_nM} \right| \leq \left| \frac{f(0,a_nMq)}{a_nM} \right| \left| 1 - \frac{a_nM}{n} \right| + \frac{1}{n} \left| f(0,a_nMq) - f(0,nq) \right|\ . 
\]
The first term on the right converges to 0. To show the same for the second term we use the fact that $f(x,y) \in L^1(\mu,\Omega_2)$ for all $x,y \in \mathbb{R}^2$. Indeed, the difference $f(0,a_nMq)-f(0,nq)$ is equal to $f(nq, a_nMq)$, which has the same distribution as $f(0,(a_nM-n)q)$. For each $\e>0$,
\[
\sum_{n \geq 1} \mu(|f(0,(n-a_nM)q)| \geq \e n) \leq \frac{1}{\e} \sum_{i=1}^M \| f(0,-iq)\|_{L^1(\mu)} < \infty\ .
\]
So only finitely many of the events $\{|f(0,a_nMq) - f(0,nq)| \geq \e n\}$ occur and we are done.
\end{proof}

The last proposition says that for each $q$ there exists a random variable $\rho_q = \rho(q,\Theta)$ such that $\mu$-almost surely, the above limit equals $\rho_q$. Assume now that we fix $\Theta$ such that this limit exists for all $q \in \mathbb{Q}^2$. We will consider $\rho_q$ as a function of $q$. The next theorem states that $\rho_q$ represents a random projection onto a vector $\varrho$.
\begin{thm}\label{thm: pizzapie}
There exists a random vector $\varrho = \varrho(\Theta)$ such that
\[
\mu\left( \rho_q = \varrho \cdot q \text{ for all } q \in \mathbb{Q}^2 \right) = 1\ .
\]
Furthermore $\varrho$ is translation invariant:
\[
\varrho(\tilde T_m \Theta) = \varrho(\Theta) \text{ for } m=1,2\ .
\]
\end{thm}

\begin{proof}
We will show that $q \mapsto \rho_q$ is a (random) linear map on $\mathbb{Q}^2$. Specifically, writing an arbitrary $q \in \mathbb{Q}^2$ as $(q_1,q_2)$, we will show that
\begin{equation}\label{eq: blackboard2}
\mu\left(\rho_q = q_1 \rho_{\unitv_1} + q_2 \rho_{\unitv_2} \text{ for all } q \in \mathbb{Q}^2\right) = 1\ .
\end{equation}
Then, setting $\varrho = (\rho_{\unitv_1},\rho_{\unitv_2})$, we will have proved the theorem.

The first step is to show translation invariance of $\rho_q$. Given $q \in \mathbb{Q}^2$, let $M \in \mathbb{N}$ be such that $Mq \in \mathbb{Z}^2$. For $m=1,2$, translation covariance implies
\begin{eqnarray*}
|f(0,nMq)(\tilde T_m \Theta) - f(0,nMq)(\Theta)| &=& |f(-\unitv_m, nMq-\unitv_m)(\Theta) - f(0,nMq)(\Theta)| \\
&\leq& |f(-\unitv_m,0)(\Theta)| + |f(nMq-\unitv_m,nMq)(\Theta)|\ .
\end{eqnarray*}
Furthermore, given $\delta>0$,
\[
\sum_n \mu\left( |f(nMq-\unitv_m, nMq)| > \delta n\right) \leq \sum_n \mu\left( |f(0, \unitv_m)| > \delta n\right) \leq \frac{1}{\delta} \|f(0,\unitv_m)\|_{L^1(\mu)} < \infty\ .
\]
Therefore only finitely many of the events $\{|f(nMq - \unitv_m,nMq)| > \delta n\}$ occur and 
\[
\rho_q(\tilde T_m\Theta) = \lim_{n \to \infty} \frac{f(0,nMq)(\tilde T_m \Theta)}{nM} = \lim_{n \to \infty} \frac{f(0,nMq)(\Theta)}{nM} = \rho_q(\Theta) \text{ almost surely}\ .
\]

To complete the proof we show that $q \mapsto \rho_q$ is almost surely additive. Over $\mathbb{Q}$, this suffices to show linearity and thus \eqref{eq: blackboard2}. Let $q_1,q_2 \in \mathbb{Q}^2$ and choose $M \in \mathbb{N}$ with $Mq_1,Mq_2 \in\mathbb{Z}^2$. By Proposition~\ref{prop: rationallimits}, for $\e>0$, we can pick $N$ such that if $n \geq N$ then the following hold:
\begin{enumerate}
\item $\mu\left( |(1/nM)f(0,nMq_1) - \rho_{q_1} | > \e/2 \right) < \e/2$ and
\item $\mu\left( |(1/nM)f(0,nMq_2) - \rho_{q_1}| > \e/2 \right) < \e/2$.
\end{enumerate}
Writing $\tilde T_{-q}(\Theta)(x) = \Theta(x+Mq)$ and using translation invariance of $\rho_{q_2}$,
\begin{eqnarray*}
&& f(0,nM(q_1+q_2))(\Theta) - nM\rho_{q_1}(\Theta) - nM\rho_{q_2}(\Theta) \\
&=& f(0,nMq_1)(\Theta) - nM\rho_{q_1}(\Theta) + f(0,nMq_2)(\tilde T_{-q_1}^n \Theta) - nM\rho_{q_2}(\tilde T_{-q_1}^n \Theta)\ .
\end{eqnarray*}
So by translation invariance of $\mu$ and items 1 and 2 above,
\begin{eqnarray*}
&& \mu(|(1/nM)f(0,nM(q_1+q_2)) - (\rho_{q_1}+\rho_{q_2})| > \e) \\
&\leq& \mu(|(1/nM)f(0,nMq_1) - \rho_{q_1}| > \e/2) + \mu(|(1/nM)f(0,nMq_2) - \rho_{q_2}| > \e/2) < \e\ .
\end{eqnarray*}
Thus $(1/nM)f(0,nM(q_1+q_2))$ converges in probability to $\rho_{q_1} + \rho_{q_2}$. By Proposition~\ref{prop: rationallimits}, this equals $\rho_{q_1+q_2}$.
\end{proof}

\subsection{A shape theorem} 

We will now upgrade the almost-sure convergence in each rational direction, from Proposition~\ref{prop: rationallimits}, to a sort of shape theorem for the Busemann function $f$. The major difference is that, unlike in the usual shape theorem of first-passage percolation, the limiting shape of $f$ is allowed to be random.

\begin{thm}
\label{shapetheorem}
For each $\delta>0$,
\begin{equation}
\label{shape_condition}
\mu\left(|f(0,x) - x \cdot \varrho| < \delta \|x\|_1 \text{ for all } x \text{ with } \|x\|_1 \geq M \text{ and all large } M \right) = 1.
\end{equation}
\end{thm}

As in the proofs of the usual shape theorems, we will need a lemma which allows us to compare $f$ in different directions. A result showing that with positive probability, $f(0,x)$ grows at most linearly in $\|x\|$ will be sufficient for our purposes. The fourth item of Proposition \ref{bd_f_by_tau_prop} allows us to derive such a bound by comparison with the usual passage time $\tau(0,x).$

\begin{lem}
There exist deterministic $K < \infty$ and $p_g > 0$ depending only on the passage time distribution such that
\[\passage\left( \sup_{\substack{x \in \Z2\\x \neq 0}} \frac{\tau(0,x)}{\|x\|_1} \leq K\right) = p_g > 0.  \]
\end{lem}
\begin{proof}
By the first-passage shape theorem, there exists $\lambda < \infty$ and $T, p_g > 0$ such that
\[
\passage\left(\forall t \geq T, \,B(t)/t \supseteq [-\lambda,\lambda]^2\,\right) = p_g\ .
\]
(Here we are using \eqref{eq: normequivalence}.) Choosing $K = T + 2 / \lambda$ completes the proof.

\end{proof}

The development of the shape theorem from this point is similar to that of the usual first-passage shape theorem for ergodic passage time distributions.  

We will say that $z \in \mathbb{Z}^2$ is ``good" for a given outcome if
\[ 
\sup_{\substack{x \in \Z2\\x \neq z}} \frac{\tau(z,x)}{\|x - z\|_1} \leq K\ . 
\]
Note that $\passage(z \text{ is good}) = p_g>0$ for all $z \in \mathbb{Z}^2.$

\begin{lem}
\label{cheesy_lemma}
Let $\zeta$ be a nonzero vector with integer coordinates, and let $z_n =n\zeta.$ Let $(n_k)$ denote the increasing sequence of integers such that $z_{n_k}$ is good. $\passage$-almost surely, $(n_k)$ is infinite and $\lim_{k \rightarrow \infty} (n_{k+1}/n_k) = 1$.
\end{lem}
\begin{proof}
The ergodic theorem shows that $(n_k)$ is a.s. infinite.  Let $B_i$ denote the event that $z_i$ is good.
By another application of the ergodic theorem,
\begin{equation} 
\label{cheesier_still}
\frac{k}{n_k} = \frac{1}{n_k} \sum_{i=1}^{n_k} \mathbf{1}_{B_i} \longrightarrow p_g \quad \text{a.s.}
\end{equation}
Thus,
\[\frac{n_{k+1}}{n_k} = \left(\frac{n_{k+1}}{k+1}\right) \left(\frac{k}{n_k}\right) \left(\frac{k+1}{k}\right) \longrightarrow 1 \quad \text{a.s.},\]
since the first and second factors converge to $p_g$ and $p_g^{-1}$ by (\ref{cheesier_still}).

\end{proof}

In what follows, we will use the fact that there is a positive density of good sites to show convergence of $f(0,z) / \|z\|_1$ in all directions. Given the convergence of $f(0,nq) / n$ for each rational $q,$ we will find enough good sites along lines close to $nq$ to let us to bound the difference $|f(0,nq) - f(0,z)|.$ To describe this procedure, we need to make several definitions. Call a vector $\zeta$ satisfying the a.s. event of Lemma \ref{cheesy_lemma} a good direction. We will extend this definition to $\zeta \in \mathbb{Q}^2$: such a $\zeta$ will be called a good direction if $m\zeta$ is, where $m$ is the smallest natural number such that $m\zeta \in \mathbb{Z}^2$. 

By countability, there exists a probability one event $\Omega''$ on which each $\zeta \in \mathbb{Q}^2$ is a good direction.
For each integer $M \geq 1,$ let $V_M = \left\{ x/M: x \in \Z2\right\},$ and let $V = \cup_{M \geq 1} V_M.$
Set $B = \{z \in \mathbb{R}^2: z \in V, \, \|z\|_1 = 1\}$ and note that $B$ is dense in the unit sphere of $\mathbb{R}^2$ (with norm $\|\cdot\|_1$).
By Theorem \ref{thm: pizzapie}, we can find a set $\hat{\Omega} \subseteq \Omega_2$ with $\mu (\hat{\Omega}) = 1$ such that, for all $\Theta \in \hat{\Omega},$
\begin{equation}\label{eq: planetolondon}
\lim_{n \rightarrow \infty} \frac{1}{n} f(n z_0) (\Theta) = z_0 \cdot \varrho(\Theta) \text{ for all } z_0 \in B\ .
\end{equation}

\begin{proof}[Proof of Theorem \ref{shapetheorem}]
Assume that there exist $\delta > 0$ and an event $D_\delta$ with $\mu(D_\delta)>0$ such that, for every outcome in $D_\delta,$ there are infinitely many vertices $x \in \Z2$ with $\left|f(x) - x \cdot \varrho\right| \geq \delta \|x\|_1.$ Then $D_\delta \cap\hat{\Omega} \cap \Omega''$ is nonempty and so it contains some outcome $(\omega,\Theta,\eta)$. We will derive a contradiction by showing that $(\omega,\Theta,\eta)$, by way of its membership in these three sets, has contradictory properties.

By compactness of the $\ell^1$ unit ball, we can find a sequence $\{x_n\}$ in $\Z2$ with $\|x_n\| \rightarrow \infty$ and $y \in \mathbb{R}^2$ with $\|y\|_1 = 1$ such that $x_n/\|x_n\|_1 \to y$ and 
\begin{equation}
\label{non_converging_pizza}
\left| \frac{f(x_n)[\Theta]}{\|x_n\|_1} - y \cdot \varrho[\Theta] \right| > \frac{\delta}{2} \text{ for all } n\ .
\end{equation}
Let $\delta' > 0$ be arbitrary (we will ultimately take it to be small). Our first goal is the approximation of $x_n$ by multiples of some element of $B.$
Choose $z \in B$ such that $\|z - y\|_1 <\delta'$ and let $\{n_k\}$ denote the increasing sequence of integers such that $n_k z$ is good. (Here if $z \notin \mathbb{Z}^2$, then $z$ being good means that $Mz$ is good, where $Mz$ was chosen after Lemma~\ref{cheesy_lemma} to be $\mathbb{Z}^2$. Therefore $(n_k)$ would then be of the form $(M l_k)$ for some increasing sequence $l_k$.) Note that $n_{k+1}/n_k \to 1$ by Lemma \ref{cheesy_lemma} so we are able to choose a $K > 0$ such that 
\begin{equation}\label{eq: planetolondon2}
n_{k+1} < (1+\delta') n_k \text{ and } \left|\frac{f(0,n_k z)}{n_k} - \varrho \cdot z\right| \leq \delta' \text{ for all } k > K\ .
\end{equation}

By the triangle inequality, the left-hand side of (\ref{non_converging_pizza}) is bounded above by
\begin{align}
\label{pizza_telescope}
 \left|\frac{f(0,x_n)}{\|x_n\|_1} - \frac{f(0,n_k z)}{\|x_n\|_1} \right| +  \left|\frac{f(0,n_k z)}{\|x_n\|_1} - \frac{f(0,n_k z)}{n_k}\right| + \left| \frac{f(0,n_k z)}{n_k} - \varrho \cdot z  \right| + \left| \varrho \cdot z - \varrho \cdot y\right|
\end{align}
for arbitrary $n$ and $n_k.$ Choose some $N_0$ such that
$\|x_n - \|x_n\|_1\,y\|_1 \leq \delta' \|x_n\|_1$ for all $n > N_0,$ and note that
\begin{equation}\label{ex_to_zee}
 \left\|x_n - \|x_n\|_1 z\right\|_1 \leq \left\|x_n - \|x_n\|_1 y \right\|_1 + \|x_n\|_1\left\| y - z \right\|_1 \leq 2 \|x_n\|_1 \delta' \text{ for } n > N_0\ .
\end{equation}
For any $n$, let $k=k(n)$ be the index such that $n_{k+1} \geq \|x_n\|_1 > n_k.$ If $n$ is so large that $k(n) > K$, then $\|\, \|x_n\|_1 z - n_k z\|_1 < \delta' \|x_n\|_1.$ Combining this observation with (\ref{ex_to_zee}) gives
\begin{equation}
\label{ex_to_zee_2}
\|x_n - n_k z \|_1 \leq 3 \delta' \|x_n\|_1 \text{ for } \|x_n\|_1 \in (n_k,n_{k+1}] \text{ when } k = k(n) > K\ .
\end{equation}

For the remainder of the proof, fix any $n > N_0$ such that $k = k(n) > K$, so that (\ref{ex_to_zee_2}) holds. We will now control the terms in (\ref{pizza_telescope}), working our way from right to left. The rightmost term may be bounded by noting
\[ 
|  \varrho \cdot z - \varrho \cdot y| = | \varrho \cdot (z - y) | \leq  \|z-y\|_2 \|\varrho\|_2 \leq \delta' \|\varrho\|_2\ .
\]
The second term from the right is bounded above by $\delta'$ by \eqref{eq: planetolondon2}. To bound the third term from the right, note that $n_k < \|x_n\|_1 \leq n_{k+1}$, so by \eqref{eq: planetolondon2},
\begin{align*}
\left|\frac{f(0,n_k z)}{\|x_n\|_1} - \frac{f(0,n_k z)}{n_k}\right| &= \left|\frac{f(0,n_k z)}{n_k}\right| \left(1 - \frac{n_k}{\|x_n\|_1} \right) \,\\
&\leq \left[\left| \varrho \cdot z\right| +\delta' \right] \left( 1- \frac{1}{1 + \delta'}\right)\ .
\end{align*}
It remains to bound the first term of (\ref{pizza_telescope}). To do this, note that by \eqref{ex_to_zee_2},
\[
|f(0,x_n) - f(0,n_k z)| = |f(n_k z, x_n)| \leq \tau(n_k z, x_n) \leq K \|x-n_k z\|_1 \leq 3K \delta' \|x_n\|_1\ .
\]
So 
\[ 
\left|\frac{f(0,x_n)}{\|x_n\|_1} - \frac{f(0,n_k z)}{\|x_n\|_1}\right| \leq 3K \delta'\ .
\]

Applying our estimates for each term in (\ref{pizza_telescope}) to the left side of (\ref{non_converging_pizza}) gives
\[
\frac{\delta}{2} \leq 3K \delta' + (|\varrho \cdot z|+\delta') \left( 1- \frac{1}{1 + \delta'}\right) + \delta' + \delta' \|\varrho\|_2\ .
\]
Because this holds for all $\delta' >0,$ and because the right-hand side goes to zero as $\delta' \rightarrow 0,$ we have derived a contradiction and proved the theorem.
\end{proof}

\subsection{General properties of $\varrho$}

In this short section we study the random vector $\varrho$. In the case that $\partial \mathcal{B}$ is differentiable at $\varpi$, the vector $\varrho$ is deterministic and we give the explicit form. 

The main theorem of the section is below. It says that the line
\[
L_\varrho : = \{x \in \mathbb{R}^2 : \varrho \cdot x =1\}
\]
is $\mu$-almost surely a supporting line for $\mathcal{B}$ at $\varpi$. 

\begin{thm}
\label{rational_directions}
With $\mu$-probability one, $\varrho \cdot \varpi = 1$ and $\varrho \cdot x \leq 1$ for all $x \in \mathcal{B}$. Thus $L_\varrho$ is a supporting line for $\mathcal{B}$ at $\varpi$.
\end{thm}

This theorem has an important corollary. It follows directly from the fact that there is a unique supporting line for $\mathcal{B}$ at points of differentiability of $\partial \mathcal{B}$.
\begin{cor}
\label{cor: diff}
If $\partial \mathcal{B}$ is differentiable at $\varpi$ then
\[ 
\mu\big( \varrho = (g_\varpi(\unitv_1),g_\varpi(\unitv_2)) \big) = 1\ .
\]
\end{cor}

\begin{proof}[Proof of Theorem~\ref{rational_directions}]

Using Theorem~\ref{thm: expected_value}, we first find the expected value of $\varrho \cdot y$ for $y \in \mathbb{R}^2$. We simply apply the dominated convergence theorem with the bound $|f(0,my)| \leq \tau(0,my)$. Letting $y_m \in \mathbb{Z}^2$ be such that $my \in y_m + [-1/2,1/2)^2$,
\[
\mathbb{E}_\mu (\varrho \cdot y) = \lim_{m \to \infty} \frac{1}{m}\mathbb{E}_\mu f(0,my) = \lim_{m \to \infty} g_\varpi(y_m/m) = g_\varpi(y)\ .
\]
The theorem follows from this statement and
\begin{equation}\label{eq: nachosgrande}
\mu\left(x \cdot \varrho \leq g(x) \text{ for all } x \in \mathcal{B}\right)=1\ .
\end{equation}
Indeed, assuming this, we have
\[
\mu(\varrho\cdot \varpi \leq 1)=1 \text{ and } \mathbb{E}_\mu(\varrho \cdot \varpi) = g_\varpi(\varpi) = 1\ ,
\]
giving $\varrho \cdot \varpi = 1$ with $\mu$-probability one. To prove \eqref{eq: nachosgrande}, first take $x \in \mathbb{Q}^2 \cap \mathcal{B}$. Then by \eqref{eq: fboundtau}, for all $n$, $f(nx) \leq \tau(nx)$ with $\mu$-probability one. Dividing by $n$ and taking limits with Proposition~\ref{prop: rationallimits} and the shape theorem we get $x \cdot \varrho \leq g(x)$. For non-rational $x \in \mathcal{B}$ we extend the inequality by almost sure continuity of both sides in $x$.
\end{proof}

\section{Geodesic graphs}\label{sec: GG}

In this section we study the behavior of $\mu$ on $\Omega_3$. Given $\eta \in \Omega_3$ recall from Section~\ref{sec: mudef} the definition of the geodesic graph $\mathbb{G}$ of $\eta$ as the directed graph induced by the edges $e$ for which $\eta(e)=1$. In this section we prove a fundamental property about infinite directed paths in this graph which relates them to the asymptotic Busemann function constructed from $\Theta$.

\subsection{Basic properties}

We begin by showing that properties of $\eta_\alpha$ from Section~\ref{subsec: GG} carry over to $\eta$. We use some new notation. We say that $y\in \mathbb{Z}^2$ is connected to $z \in \mathbb{Z}^2$ in $\mathbb{G}$ (written $y \to z$) if there exists a sequence of vertices $y=y_0, y_1, \ldots, y_n = z$ such that $\eta(\langle y_k,y_{k+1}\rangle) = 1$ for all $k=0, \ldots, n-1$. We say that a path in $\mathbb{G}$ is a geodesic (for the configuration $(\omega,\Theta,\eta)$) if it is a geodesic in $\omega$.

\begin{prop}\label{prop: firstGG2}
With $\mu$-probability one, the following statements hold for $x,y,z \in \Z2$.
\begin{enumerate}
\item Each directed path in $\mathbb{G}$ is a geodesic.
\item If $x \to y$ in $\mathbb{G}$ then $f(x,y) = \tau(x,y)$.
\item If $x \to z$ and $y \to z$ in $\mathbb{G}$ then $f(x,y) = \tau(x,z) - \tau(y,z)$.
\item There exists an infinite self-avoiding directed path starting at $x$ in $\mathbb{G}$.
\end{enumerate}
\end{prop}

\begin{proof}
The third item follows directly from the second and additivity of $f$ (from \eqref{eq: additivity}). For the first item, if $\gamma$ is a deterministic finite directed path, write $A_\gamma$ for the event that all edges of $\gamma$ are edges of $\mathbb{G}$ and
\[
B_\gamma = A_\gamma^c \cup \left( A_\gamma \cap \{\gamma \text{ is a geodesic}\}\right)\ .
\]
The event in question equals the intersection over all finite $\gamma$'s of $B_\gamma$, so it suffices to show that for each $\gamma$, $\mu(B_\gamma)=1$.

By part 1 of Proposition~\ref{prop: firstGG}, for all $\alpha \in \mathbb{R}$ the $\mathbb{P}$-probability that all directed paths in $\mathbb{G}_\alpha(\omega)$ are geodesics is 1. By pushing forward to $\widetilde \Omega$, for each $\alpha,~ \mu_\alpha(B_\gamma) = 1$ and thus $\mu_n^*(B_\gamma) = 1$ for all $n$. Once we show that $B_\gamma$ is a closed event, we will be done, as we can then apply \eqref{eq: kallenbergclosed}. To show this we note that the event that a given finite path is a geodesic is a closed event. Indeed, letting $\gamma_1$ and $\gamma_2$ be finite paths, the function $\tau(\gamma_1) - \tau(\gamma_2)$ is continuous on $\widetilde \Omega$. Therefore the event $\{\omega \in \Omega_1 : \tau(\gamma_1) \leq \tau(\gamma_2)\}$ is closed. We then write
\[
\{\gamma_1 \text{ is a geodesic}\} = \bigcap_{\gamma_2} \{\tau(\gamma_1) \leq \tau(\gamma_2)\}\ ,
\]
where the intersection is over all finite paths $\gamma_2$ with the same endpoints as those of $\gamma_1$. Thus $\{\gamma_1 \text{ is a geodesic}\}$ is closed. Since $A_\gamma$ depends on finitely many edge variables $\eta(e)$, it is closed and its complement is closed. Therefore $B_\gamma$ is closed and we are done.

For item 2, we write $\gamma_{xy}$, any path from $x$ to $y$ in $\mathbb{G}$, in order as $x=x_0, x_1, \ldots, x_n = y$ and use additivity of $f$:
\[
f(x,y) = \sum_{i=0}^{n-1} f(x_i,x_{i+1})\ .
\]
For each $i$, $x_i \to x_{i+1}$, and by item 1, $\gamma_{xy}$ is a geodesic. This means that we only need to show that if $x$ and $y$ are neighbors such that $\eta(\langle x,y \rangle) = 1$ then $f(x,y) = \omega_{\langle x,y \rangle}$, the passage time of the edge between $x$ and $y$. By part 2 of Proposition~\ref{prop: firstGG}, for each $\alpha$, with $\mathbb{P}$-probability one, if $\eta_\alpha(\langle x,y \rangle) = 1$ then $B_\alpha(x,y) = \omega_{\langle x,y \rangle}$. By similar reasoning to that in the last item,
\[
\{\eta(\langle x,y \rangle) = 0\} \cup \left( \{ \eta(\langle x,y \rangle) = 1\}  \cap \{f(x,y) = \omega_{\langle x,y \rangle} \} \right)
\]
is closed and since it has $\mu_\alpha$-probability 1 for all $\alpha$, it also has $\mu$-probability one.

We now argue for item 4. By translation-invariance we can just prove it for $x=0$. For $n \geq 1$ let $A_n \subseteq \Omega_3$ be the event that there is a self-avoiding directed path starting at 0 in $\mathbb{G}$ that leaves $[-n,n]^2$. We claim that $\mu(A_n) = 1$ for all $n$. Taking $n \to \infty$ will prove item 4. 

For each $\alpha>0$ so large that $[-n,n]^2$ is contained on one side of $L_\alpha$, let $\gamma$ be a geodesic from $0$ to $L_\alpha$. This path is contained in $\mathbb{G}_\alpha$. We may remove loops from $\gamma$ so that it is self-avoiding, and still a geodesic. It will also be directed in the correct way: as we traverse the path from 0, each edge will be directed in the direction we are traveling. So for all large $\alpha>0$, with $\mathbb{P}$-probability one, there is a self-avoiding directed path starting at 0 in $\mathbb{G}_\alpha$ that leaves $[-n,n]^2$. Thus $\mu_\alpha(A_n) = 1$ for all large $\alpha$ and $\mu_{n_k}^*(A_n) \to 1$ as $k \to \infty$. The indicator of $A_n$ is continuous on $\widetilde \Omega$, as $A_n$ depends on $\eta(f)$ for finitely many edges $f$, so $\mu(A_n)=1$.
\end{proof}

\begin{prop}\label{prop: secondGG2}
Assume {\bf A1'} or {\bf A2'}. With $\mu$-probability one, the following statements hold.
\begin{enumerate}
\item Each vertex in $\Z2$ has out-degree 1 in $\mathbb{G}$. Consequently from each vertex $x$ emanates exactly one infinite directed path $\Gamma_x$.
\item Viewed as an undirected graph, $\mathbb{G}$ has no circuits. 
\end{enumerate}
\end{prop}

\begin{proof}
For $x \in \mathbb{Z}^2$, let $A_x\subseteq \widetilde \Omega$ be the event that $\eta(\langle x,y\rangle) = 1$ for only one neighbor $y$ of $x$. Note that the indicator of $A_x$ is a bounded continuous function, so since $\mu_\alpha(A_x) = 1$ for all $\alpha$ such that $x$ is not within Euclidean distance $1$ of $L_\alpha$ (from part 1 of Proposition~\ref{prop: secondGG} -- here $\hat S$ is contained in the set of vertices within distance 1 of $L_\alpha$) it follows that $\mu(A_x)=1$. For each $z$ that is not a neighbor of $x$, $\eta(\langle x,z \rangle)=0$ with $\mu_\alpha$-probability one for all $\alpha$. This similarly implies that in $\mathbb{G}$ with $\mu$-probability one, there is no edge between $x$ and such a $z$.

To prove the second statement, fix any circuit $\mathcal{C}$ in $\mathbb{Z}^2$ and let $A_\mathcal{C}$ be the event that each edge of $\mathcal{C}$ is in $\mathbb{G}$. Because there are no circuits in $\mathbb{G}_\alpha$ with $\mathbb{P}$-probability one, we have $\mu_n^*(A_\mathcal{C})=0$ for all $n$. The indicator of $A_\mathcal{C}$ is a continuous function on $\widetilde \Omega$, so we may take limits and deduce $\mu(A_\mathcal{C})=0$. There are a countable number of circuits, so we are done. 
\end{proof}

\subsection{Asymptotic directions}
Recall the definition $L_\varrho = \{x \in \mathbb{R}^2: x \cdot \varrho = 1\}$ for the vector $\varrho = \varrho(\Theta)$ of Theorem~\ref{thm: pizzapie}. Set 
\begin{equation}\label{eq: on_skype}
J_\varrho = \{ \theta : L_\varrho \text{ touches } \mathcal{B} \text{ in direction } \theta\}\ .
\end{equation}
The main theorem of this subsection is as follows.
\begin{thm}\label{thm: nachostheorem}
With $\mu$-probability one, for all $x \in \Z2$, the following holds. Each directed infinite self-avoiding path in $\mathbb{G}$ which starts at $x$ is asymptotically directed in $J_{\varrho}$.
\end{thm}


\begin{proof}
We will prove the theorem for $x=0$. Assuming we do this, then using translation invariance of $\mu$ and $\varrho$ it will follow for all $x$.

Let $\e_k = 1/k$ for $k \geq 1$ and $\delta>0$. We will show that if $S_0= \{x \in \mathbb{Z}^2 : 0 \to x$ in $\mathbb{G}\}$ then
\begin{equation}\label{eq: nachos1}
\text{ for each } k \geq 1,~ \mu(\arg x \in (J_\varrho)_{\e_k} \text{ for all but finitely many }  x \in S_0) > 1-\delta\ .
\end{equation}
Here we write $(J_\varrho)_{\e_k}$ for all angles $\theta$ with $dist(\theta,\theta')<\e_k$ for some $\theta' \in J_\varrho$. The line $L_\varrho$ only touches $\mathcal{B}$ in directions in $J_\varrho$ so by convexity, $v_\theta \cdot \varrho < 1$ for all $\theta \notin J_\varrho$. Since the set of angles not in $(J_\varrho)_{\e_k}$ is compact in $[0,2\pi)$ (using the metric $dist$), we can find a random $a \in (0,1)$ with $v_\theta \cdot \varrho < 1-a$ for all $\theta \notin (J_\varrho)_{\e_k}$. We can then choose $a$ to be deterministic such that
\begin{equation}\label{eq: nachos2}
\mu\left( v_\theta \cdot \varrho < 1-a \text{ for all } \theta \notin (J_\varrho)_{\e_k} \right) > 1-\delta/3\ .
\end{equation}

By the shape theorem there exists $M_0$ such that $M \geq M_0$ implies
\[
\mathbb{P}( \tau(0,x) \geq g(x)(1-a/2)  \text{ for all } x \text{ with } \|x\|_1 \geq M) > 1-\delta/3\ .
\]
The marginal of $\mu$ on $\Omega_1$ is $\mathbb{P}$ so this holds with $\mu$ in place of $\mathbb{P}$. By part 2 of Proposition~\ref{prop: firstGG2},  
\begin{equation}\label{eq: nachos3}
\mu( f(x) \geq g(x)(1-a/2) \text{ for all } x \text{ with } \|x\|_1 \geq M \text{ and } 0 \to x) > 1-\delta/3\ .
\end{equation}
Choose $C>0$ such that $\|x\|_1 \leq Cg(x)$ for all $x \in \mathbb{R}^2$. This is possible by \eqref{eq: normequivalence}. By Theorem~\ref{shapetheorem}, there exists $M_1 \geq M_0$ such that $M \geq M_1$ implies
\[
\mu\left( |f(x) - x\cdot\varrho|< \frac{a}{2C} \|x\|_1 \text{ for all } x \text{ with } \|x\|_1 \geq M \right) > 1-\delta/3\ .
\]
This implies that for $M \geq M_1$,
\begin{equation}\label{eq: nachos4}
\mu\left( |f(x) - x\cdot\varrho| < \frac{a}{2} g(x) \text{ for all } x \text{ with } \|x\|_1 \geq M \right) > 1-\delta/3\ .
\end{equation}

We claim that the intersection of the events in \eqref{eq: nachos2}, \eqref{eq: nachos3} and \eqref{eq: nachos4} implies the event in \eqref{eq: nachos1}. Indeed, take a configuration in the intersection of the three events for some $M \geq M_1$. For a contradiction, assume there is an $x \in S_0$ with $\arg x \notin (J_\varrho)_{\e_k}$ and $\|x\|_1 \geq M$. Then 
\[
(x/g(x)) \cdot \varrho < 1-a \text{ by \eqref{eq: nachos2}}\ .
\]
However, since the event in \eqref{eq: nachos3} occurs and $\|x\|_1 \geq M$,
\[
f(0,x) \geq g(x)(1-a/2)\ .
\]
Last, as the event in \eqref{eq: nachos4} occurs,
\[
f(0,x) < x \cdot \varrho + \frac{a}{2} g(x)\ .
\]
Combining these three inequalities,
\[
g(x)(1-a/2) \leq x \cdot \varrho + (a/2)g(x) < g(x)(1-a) + (a/2)g(x)\  ,
\]
or $g(x)(1-a/2) < g(x)(1-a/2)$, a contradiction. This completes the proof.

\end{proof}

\section{Coalescence in $\mathbb{G}$}\label{sec: coalesceG}

In this section we prove that all directed infinite paths coalesce in $\mathbb{G}$. Recall that under either {\bf A1'} or {\bf A2'}, for $x \in \mathbb{Z}^2$, $\Gamma_x$ is the unique infinite directed path in $\mathbb{G}$ starting at $x$.

\begin{thm}\label{thm: Gcoalescethm}
Assume either {\bf A1'} or both {\bf A2'} and the upward finite energy property. With $\mu$-probability one, for each $x,y \in \mathbb{Z}^2$, the paths $\Gamma_x$ and $\Gamma_y$ coalesce.
\end{thm}

The proof will be long, so we first explain the main ideas. We apply the technique of Licea-Newman \cite{LN}, whose central tool is a Burton-Keane type argument \cite{burtonkeane}. We proceed by contradiction, so suppose there are vertices $x, y$ such that $\Gamma_x$ and $\Gamma_y$ do not coalesce. By results of the last section, they cannot even intersect. We show in Sections~\ref{sec: building_blocks} and~\ref{sec: Bprime} that there are many triples of non-intersecting paths $\Gamma_{x_1}, \Gamma_{x_2}$ and $\Gamma_{x_3}$ such that $\Gamma_{x_2}$ is ``shielded'' from all other infinite paths in $\mathbb{G}$. To do this, we must use the information in Theorem~\ref{thm: nachostheorem} about asymptotic directions. A contradiction comes in Section~\ref{sec: contradiction} from translation invariance because when $\Gamma_{x_2}$ is shielded, the component of $x_2$ in $\mathbb{G}$ has a unique least element in a certain lexicographic-like ordering of $\mathbb{Z}^2$. This is a different concluding argument than that given in \cite{LN}, where these shielded paths are used for a Burton-Keane ``lack of space'' proof.


We now give the proof. For the entirety we will assume either {\bf A1'} or both {\bf A2'} and the upward finite energy property.

\subsection{Constructing ``building blocks''}\label{sec: building_blocks}
Assume for the sake of contradiction that there are disjoint $\Gamma_x$'s in $\mathbb{G}$. Then for some vertex $z_0,$ the event $A_0(z_0)\subseteq \widetilde \Omega$ has positive $\mu$-probability, where 
\[
A_0(z_0) = \{\Gamma_{z_0} \text{ and } \Gamma_0 \text{ share no vertices}\}\ .
\]  
We begin with a geometric lemma. It provides a (random) line such that with probability one, any path that is asymptotically directed in $J_\varrho$ (from \eqref{eq: on_skype}) intersects this line finitely often. We will need some notation which is used in the rest of the proof.

Let $\varpi'$ be a vector with 
\begin{equation}\label{eq: varpiprimedef}
\arg \varpi' \in \{j \pi / 4,\, j = 0, \ldots,\,7\} \text{ and } \|\varpi'\|_\infty = 1\ ,
\end{equation}
where $\|\cdot\|_\infty$ is the $\ell^\infty$ norm. (A precise value of $j$ will be fixed shortly.) 
Define (for $N \in \mathbb{N}$) $L'_N = \{z \in \mathbb{R}^2: \, \varpi' \cdot z = N\}.$ For such an $N$ and for $x \in \Z2,$ write $x \prec L'_N$ if $\varpi' \cdot x < N$ and $x \succ L'_N$ if $\varpi' \cdot x > N.$ The symbols $\preceq$ and $\succeq$ are interpreted in the obvious way. We use the terms ``far side of $L'_N$" and ``near side of $L'_N$" for the sets of $x \in \mathbb{R}^2$ with $x \succ L'_N$ and $x \prec L'_N,$ respectively.  Note that any lattice path $\gamma$ intersecting both sides of $L'_N$ contains a vertex $z \in L'_N.$

\begin{lem}\label{lem: varpilemma}
There is a measurable choice of $\varpi'$ as in \eqref{eq: varpiprimedef} such that with $\mu$-probability one, the following holds. For each vertex $x$ and each integer $N$, 
\[
\Gamma_x \cap \{z \in \mathbb{Z}^2 : z \preceq L'_N\} \text{ is finite}\ .
\] 
In other words, $\Gamma_x$ eventually lies on the far side of $L'_N$ for all $x$ and $N$.
\end{lem}

\begin{proof}
The limit shape $\mathcal{B}$ is convex and compact, so it has an extreme point $p$. Because it is symmetric with respect to the rotation $R$ of $\mathbb{R}^2$ by angle $\pi/2$, the points $p_i = R^i p$, $i=1, \ldots, 3$ are all extreme points of $\mathcal{B}$. $J_\varrho$ is an interval of angles corresponding to points of contact between $\mathcal{B}$ and one of its supporting lines, so it is connected (in the topology induced by $dist$) and must lie between (inclusively) $\arg p_i$ and $\arg p_{i+1}$ for some $i=0, \ldots, 3$ (here we identify $p_4=p_0$). Therefore $\mathrm{diam} ~J_\varrho \leq \pi/2$ almost surely and contains at most three elements of the set $\{j\pi/4 : j = 0, \ldots, 7\}$ (and they must be consecutive). Choose five of the remaining elements to be consecutive and label them $j_1\pi/4, \ldots, (j_1+4)\pi/4$. The interval $[j_1\pi/4,(j_1+4)\pi/4]$ defines a half-plane $H$ in $\mathbb{R}^2$ and since the distance between this interval and $J_\varrho$ is positive (measured with $dist$), for all sufficiently small $\e>0$, the sector
\[
\{x \in \mathbb{R}^2 : x \neq 0 \text{ and } dist(\arg x, \phi) < \e \text{ for some } \phi \in J_\varrho\}
\]
is contained in $H^c$. This implies the statement of the lemma for a (random) $\varpi'$ equal to the normal to $H$. Since $\varpi'$ can be chosen as a measurable function of $\varrho$ (which is clearly Borel measurable on $\widetilde \Omega$), we are done.
\end{proof}

For the rest of the proof, fix a deterministic $\varpi'$ as in \eqref{eq: varpiprimedef} that satisfies Lemma~\ref{lem: varpilemma} with positive probability on the event $A_0(z_0)$. (This is possible because there are only eight choices for $\varpi'$.) Let $A_0'(0,z_0)$ be the intersection of $A_0(z_0)$ and the event in the lemma. On $A_0'(0,z_0)$, $\Gamma_0$ and $\Gamma_{z_0}$ eventually cease to intersect $L'_0.$ In particular, they each have a last intersection with $L'_0.$ Since there are only countably many possible pairs of such last intersections, we see that some pair $(y, y')$ in $L'_0$ occurs with positive probability; that is, $\mu(A(y,y'))>0$, where $A(y,y')$ is defined by the conditions
\begin{enumerate}
\item[I.] $\Gamma_y \cap \Gamma_{y'} = \varnothing;$
\item[II.] $\Gamma_y$ intersects $L'_0$ only at $y$; $\Gamma_{y'}$ intersects $L'_0$ only at $y'$ and
\item[III.] $\Gamma_u \cap L_N'$ is nonempty and bounded for $u=y,y'$ and all integers $N \geq 0$.
\end{enumerate}
(Note that condition III follows directly from the preceding lemma because $\Gamma_u$ contains infinitely many vertices.) By translation invariance, there exists $z \in L_0'$ with $\mu(A(0,z))>0$.

Fix 
\begin{equation}\label{eq: chicken_alfredo}
\varsigma = \text{ a nonzero vector with the smallest integer coordinates normal to } \varpi'
\end{equation}
(it will be a rotation of either (0,1) or (1,1) by a multiple of $\pi/2$). Defining $\tilde T_\varsigma : \widetilde \Omega \to \widetilde \Omega$ as the translation by $\varsigma$ (that is, $\tilde T_1^{a_1} \circ \tilde T_2^{a_2}$, where $\varsigma = a_1\unitv_1 + a_2 \unitv_2$),
\[ 
\mathbf{1}_{A(0,z)}\left((\omega,\Theta,\eta)\right) = \mathbf{1}_{A(\varsigma,z+\varsigma)}\left(\widetilde{T}_{\varsigma}(\omega,\Theta,\eta)\right)\ .
\]
Since $\mu$ is invariant under the action of $\widetilde{T}_{\varsigma},$  the ergodic theorem implies
\begin{equation}
\label{poincare}
\frac{1}{N}\sum_{j=0}^{N-1}\mathbf{1}_{A(j\varsigma,z+j\varsigma)}\left((\omega,\Theta,\eta)\right) = \frac{1}{N}\sum_{j=0}^{N-1}\mathbf{1}_{A(0,z)}\left(\widetilde{T}_{\varsigma}^j (\omega,\Theta,\eta)\right) \rightarrow g(\omega,\Theta,\eta),
\end{equation}
where $g$ is a function in $L^1(\mu);$ the convergence is both $\mu$-almost sure and in $L^1(\mu),$ so $\int g \, \md \mu = \mu(A(0,z))>0.$ Using this in (\ref{poincare}) gives infinitely many $j$ with 
\begin{equation}
\label{two_pair}
\mu\left(A(0,z) \cap A(j\varsigma, z + j\varsigma) \right) > 0.
\end{equation}
We fix $j > \|z\|_1$ to ensure $\Gamma_{j\varsigma}$ and $\Gamma_{z + j\varsigma}$ are outside the region bounded by $L'_0,$ $\Gamma_0,$ and $\Gamma_z.$

What is the significance of the event in (\ref{two_pair})?  When it occurs, we are guaranteed that there is a line $L_0'$ and four directed paths remaining on its far side apart from their initial vertices.  We claim that at least three of them never intersect. Indeed, ordering the paths using the direction of $\varsigma$, we are guaranteed that the ``first two" paths do not intersect each other, nor do the ``last two."  But if the middle two paths ever intersect, they would merge beyond that point and the three remaining paths could not touch.

For $x_1,x_2 \in L_0'$, let $B(0,x_1,x_2)$ be the event that $\Gamma_0, \Gamma_{x_1}$ and $\Gamma_{x_2}$ (a) never intersect, (b) stay on the far side of $L_0'$ except for their initial vertices and (c) intersect $L_N'$ in a bounded set for each $N \geq 1$. Then the above implies 
\[
B(0,z,j\varsigma) \cup B(0,z,z+j\varsigma) \supseteq A(0,z) \cap A(j\varsigma,z + j\varsigma)\ .
\]
Therefore we may choose $x_1,x_2 \in L_0'$ such that the portion of $L_0'$ from 0 to $x_2$ contains $x_1$ and so that $\mu(B(0,x_1,x_2)) > 0$. The vertices $x_1$ and $x_2$ are fixed for the rest of the proof.

\subsection{Constructing $B'$}\label{sec: Bprime}

Our next step is to refine $B(0,x_1,x_2)$ to a positive probability subevent $B'(x^*;N,R)$ on which no paths $\Gamma_z$ with $z \preceq L'_N$ (outside of some large polygon) merge with $\Gamma_{x_1}.$ We will need to pull events back from $\widetilde \Omega$ to $\Omega_1$ to do an edge modification and this will present a considerable difficulty. Our strategy is reminiscent of that in \cite{AD}. In the first subsection we give several lemmas that we will need. In the next subsection we will define $B'$ and show it has positive probability.

\subsubsection{Lemmas for $B'$}

We wish to construct a barrier of high-weight edges on the near side of some $L'_N$. Set
\[ 
\lambda_0^+ = \sup \left\{ \lambda > 0: \, \passage\left( \omega_e \in [\lambda,\infty)\right) > 0 \right\}\ .
\]  
Because we do not wish to assume $\lambda_0^+ = \infty,$ our barrier will occupy some wide polygon (in the case that $\lambda_0^+ = \infty,$ many of the complications which we address below can be neglected; we direct the interested reader to \cite{LN}).
To control the exit of our directed paths from the polygon, we will need a lemma about weak angular concentration of paths:

\begin{lem}
\label{path_concentration}
For $x_1$, $x_2$, and $\varpi'$ as above, define $B_G(0,x_1,x_2)$ to be the subevent of $B(0,x_1,x_2)$ on which, for all $\varepsilon > 0,$ there are infinitely
many values of $N \in \mathbb{N}$ such that the first intersections $\zeta_N$ and $\zeta'_N$ of $\Gamma_0$ and $\Gamma_{x_2}$ (respectively) with $L'_N$ satisfy $dist(\arg \zeta_N, \arg \zeta'_N) < \varepsilon.$
Then $\mu\left(B_G (0,x_1,x_2) \mid \, B(0,x_1,x_2)\right) = 1.$
\end{lem}
\begin{proof}
Assume for the sake of contradiction that 
\begin{equation}
\label{paths_spreading}
\mu\left(B_G^c(0,x_1,x_2) \cap B(0,x_1,x_2)\right) > 0.
\end{equation}
For $z \in \mathbb{Z}^2$, denote by $\zeta_N(z)$ the first point of intersection of $\Gamma_z$ with $L'_N.$  On the event in (\ref{paths_spreading}), for all but finitely many $N \in \mathbb{N}$, we have $dist(\arg \zeta_N(0), \arg \zeta_N(x_2)) > \varepsilon / 2.$  Taking $\varsigma$ as before (fixed in \eqref{eq: chicken_alfredo}) and translating the event in (\ref{paths_spreading}) by multiples of $\varsigma$, we see by the ergodic theorem that with positive $\mu$-probability infinitely many such translates occur.

So given any finite $b > 0,$ we can find an event of positive $\mu$-probability on which we have at least $b$ directed paths in $\mathbb{G}$ which never return to $L'_0$ and such that the first intersections of neighboring paths with lines $L'_N$ stay at least an angle $\varepsilon$ apart. This is in contradiction with the fact that all directed infinite paths are asymptotically confined to a sector.
\end{proof}

The next lemma is a modification of the usual first-passage shape theorem.
\begin{lem}
\label{l1_shape}
There exists a deterministic $c^+ < \lambda_0^+$ such that, $\passage$-a.s., 
\[
\lim_{M \to \infty} \sup_{\|x\|_1 \geq M} \tau(0,x)/\|x\|_1 < c^+\ .
\]
\end{lem}

\begin{proof}
Because either {\bf A1'} or {\bf A2'} hold, $\mathbb{E}(\tau_e) < \lambda_0^+$. For any $z \in \mathbb{Z}^2$, choose a deterministic path $\gamma_z$ with number of edges equal to $\|z\|_1$. For $x \in \mathbb{Q}^2$ and $n \geq 1$ with $nx \in \mathbb{Z}^2$,
\[
\mathbb{E} \tau(0,nx) \leq \mathbb{E} \tau(\gamma_{nx}) = n \|x\|_1 \mathbb{E}\tau_e\ , \text{ so } g(x) \leq \|x\|_1 \mathbb{E} \tau_e\ .
\] 
This extends to all $x \in \mathbb{R}^2$ by continuity, so the shape theorem gives the result.
\end{proof}

We need a lemma to pull events back from $\widetilde \Omega$ to $\Omega_1$. Fix an increasing sequence $(n_k)$ such that $\mu^*_{n_k} \to \mu$ weakly.
\begin{lem}
\label{toalphas}
Let $E \subseteq \widetilde \Omega$ be open with $\mu(E) > \beta$. There exists $C_\beta>0$ and $K_0$ such that for $k\geq K_0$, the Lebesgue measure of the set $\{ \alpha \in [0,n_k]: \, \mu_\alpha (E)  > \beta/2 \}$ is at least $C_\beta\, n_k$.
\end{lem}
\begin{proof}
Call the Lebesgue measure of the above set $\lambda.$ Since $E$ is open, \eqref{eq: kallenbergopen} allows us to pick $K_0$ such that if $k \geq K_0$ then $\mu^*_{n_k}(E) > \beta$. For such $k$, we can write
\[
\frac{1}{n_k} \left( \lambda + (n_k - \lambda) \beta/2\right) \geq \mu_{n_k}^*(E) > \beta,~ \text{giving } \lambda > \frac{n_k \beta}{2 (1 - \beta/2)}\ .
\]
Setting $C_\beta := \beta (2 - \beta)^{-1}$ completes the proof.
\end{proof}

The last lemma is based on \cite[Lemma~3.4]{AD} and will be used in the edge-modification argument. To push the upward finite energy property forward from $\Omega_1$ to $\widetilde \Omega$ we need concrete lower bounds for probabilities of modified events. We write a typical element of $\Omega_1$ as $\omega = (\omega_e, \check{\omega}),$ where $\check{\omega} = (\omega_f)_{f \neq e}.$ We say an event $A \subseteq \Omega_1$ is {\it $e$-increasing} if, for all $(\omega_e,\check{\omega}) = \omega \in A$ and $r > 0,$ $(\omega_e + r, \check{\omega}) \in A.$

\begin{lem}
\label{lem: edge_modification}
Let $\lambda > 0$ be such that $\mathbb{P}\left(\omega_e \geq \lambda \right) > 0.$ For each $\vartheta>0$ there exists $C = C(\vartheta,\lambda) > 0$ such that for all edges $e$ and all $e$-increasing events $A$ with $\mathbb{P}(A) \geq \vartheta$,
\[
\mathbb{P}\left( A, ~\omega_e \geq \lambda\right) \geq C ~\mathbb{P}\left(A\right)\ .
\]
\end{lem}

\begin{proof}
If $\mathbb{P}(A,~ \omega_e < \lambda) \leq (1/2) \mathbb{P}(A)$ then 
\begin{equation}\label{eq: on_skype_again}
\mathbb{P}(A,~\omega_e \geq \lambda) \geq (1/2) \mathbb{P}(A)\ .
\end{equation}
Otherwise, we assume that
\begin{equation}\label{eq: newassumption}
\mathbb{P}(A,~ \omega_e < \lambda) \geq (1/2) \mathbb{P}(A)\ .
\end{equation}

We then need to define an extra random variable. Let $\omega_e'$ be a variable such that, given $\check{\omega}$ from $\omega \in \Omega_1$, it is an independent copy of the variable $\omega_e$. In other words, letting $\mathbb{Q}$ be the joint distribution of $(\omega, \omega_e')$ on the space $\Omega_1 \times \mathbb{R}$, for $\mathbb{Q}$-almost every $\check{\omega}$,
\begin{itemize}
\item $\omega_e'$ and $\omega_e$ are conditionally independent given $\check{\omega}$ and
\item the distributions $\mathbb{Q}(\omega_e \in \cdot \mid \check{\omega})$ and $\mathbb{Q}(\omega_e' \in \cdot \mid \check{\omega})$ are equal.
\end{itemize}
(This can be defined, for instance, by setting $\mathbb{Q}(A \times B) = \int_A \mathbb{P}(\omega_e \in B \mid \check{\omega}) ~\md \mathbb{P}(\omega)$ for Borel sets $A \subseteq \Omega_1$ and $B \subseteq \mathbb{R}$.)

We now write $\mathbb{P}(A, \omega_e \geq \lambda)$ as
\begin{align}
\mathbb{Q} [ (\omega_e, \check{\omega})\in A,\, \omega_e \in [\lambda,\infty)] & \geq \mathbb{Q}\left[(\omega_e,\check{\omega}) \in A,\, \omega_e \in [\lambda,\infty),\,\omega_e' \in [0,\lambda)\right]\nonumber\\
&=\mathbb{E}_{\mathbb{Q}} \left[ \mathbf{1}_{(\omega_e,\check{\omega}) \in A}\, \mathbf{1}_{ \omega_e \in [\lambda,\infty)}\, \mathbf{1}_{\omega_e' \in [0,\lambda)}\right]\nonumber\\
&\geq\mathbb{E}_{\mathbb{Q}} \left[ \mathbf{1}_{(\omega_e',\check{\omega}) \in A}\, \mathbf{1}_{ \omega_e \in [\lambda,\infty)}\, \mathbf{1}_{\omega_e' \in [0,\lambda)}\right]\label{switcheroo}\\
&=\mathbb{E}_{\mathbb{Q}} \left[ \mathbf{1}_{(\omega_e',\check{\omega})\in A}\, \mathbf{1}_{\omega_e'\in [0,\lambda)}\,  \mathbb{E}_{\mathbb{Q}}\left(\mathbf{1}_{\omega_e \in [\lambda,\infty)}\, \mid \check{\omega},\omega_e' \right)\right]. \label{condintready}
\end{align}
In (\ref{switcheroo}), we have used that $A$ is $e$-increasing. Using conditional independence in (\ref{condintready}),
\begin{equation}\label{eq: something_something}
\passage(A, ~\omega_e \geq \lambda) \geq \mathbb{E}_{\mathbb{Q}} \left[ \mathbf{1}_{(\omega_e',\check{\omega})\in A}\, \mathbf{1}_{\omega_e' \in [0,\lambda)}\,  \mathbb{E}_{\mathbb{Q}}\left(\mathbf{1}_{\omega_e \in [\lambda,\infty)}\, \mid \check{\omega}\right)\right]\ .
\end{equation}
By the upward finite energy property, 
\[
\mathbb{E}_{\mathbb{Q}}(\mathbf{1}_{\omega_e \in [\lambda,\infty)} \mid \check{\omega}) = \mathbb{E}(1_{\omega_e \in [\lambda,\infty)}\mid \check{\omega}) > 0 ~~\mathbb{Q} \text{-almost surely}\ ,
\]
so choose $c>0$ such that 
\[
\mathbb{Q}\left[ \mathbb{E}_{\mathbb{Q}}( \mathbf{1}_{\omega_e \in [\lambda,\infty)} \mid \check{\omega}) \geq c \right] \geq 1-(\vartheta/4)\ .
\]
Note that this choice of $c$ depends only on $\lambda$ and $\vartheta$. By \eqref{eq: newassumption} and the assumption $\mathbb{P}(A) \geq \vartheta$, the right side is at least $1-(1/2)\mathbb{P}(A,~\omega_e< \lambda)$, implying
\[
\mathbb{Q}\left[ (\omega_e',\check{\omega}) \in A, ~\omega_e'\in [0,\lambda),~ \mathbb{E}_{\mathbb{Q}}(\mathbf{1}_{\omega_e \in [\lambda,\infty)} \mid \check{\omega}) \geq c\right] \geq (1/2) \mathbb{P}(A,~\omega_e < \lambda)\ .
\]
Combining with \eqref{eq: something_something}, we find $\mathbb{P}(A,~\omega_e\geq \lambda) \geq (c/2)\mathbb{P}(A,~\omega_e < \lambda)$. We finish the proof by writing
\[
\passage(A) = \passage(A, \omega_e < \lambda) + \passage(A, \omega_e \geq \lambda)\\
\leq \left[\frac{2}{c}+1 \right] \passage(A,\omega_e \geq \lambda)\ .
\]
Observing this inequality and \eqref{eq: on_skype_again}, we set $C = \min\{1/2, c/(2+c))\}$.

\end{proof}

\subsubsection{Defining $B'$}

We begin with the definition of the ``barrier event'' $B'$. For an integer $R > N,$ let 
\[
S(R,N) = \{y \in \mathbb{Z}^2 : 0 \leq y \cdot \varpi' \leq N,~ |y \cdot \varsigma| \leq R\}\ .
\] 
For any vertex $x^* \in S(R,N) \cap L_N'$, define $B'(x^*;R,N)$ by the condition
\begin{equation}\label{eq: B_prime_def}
\text{for all } z \in \mathbb{Z}^2 \setminus S(R,N) \text{ with } z \preceq L_N',~ \Gamma_z \cap \Gamma_{x^*} = \varnothing\ .
\end{equation}

\begin{prop}
\label{B_prime}
There exist values of $R,N$ and $x^*$ such that $\mu(B'(x^*;R,N)) > 0.$
\end{prop}
Our strategy is to pull back cylinder approximations of $B(0,x_1,x_2)$ to $\Omega_1$ to find events that depend on $\mathbb{G}$ in the vicinity of $0, x_1$ and $x_2.$ We will find a subevent which is monotone increasing in the weights of edges lying in $S(R,N)$ between the pulled-back versions of $\Gamma_0$ and $\Gamma_{x_2}.$ When we look at the subevent on which all of these weights are large (``edge modification"), the pullback of $\Gamma_{x_1}$ will be unchanged (past $S(R,N)$), and no pullback of any $\Gamma_z$ can intersect it if $z \preceq L'_N$ and $z \notin S(R,N).$ We will then choose $x^*$ to be a certain point on $\Gamma_{x_1} \cap L'_N$. The constants $N$ and $R$ will be chosen to guarantee that the pullback of $\Gamma_{x_1}$ is so isolated. Pushing forward the subevent to $\widetilde{\Omega}$ will complete the proof.

\begin{proof}
We will first fix some parameters to prepare for the main argument. Recall the definition of $c^+$ from Lemma~\ref{l1_shape} and let
\[ 
\lambda^+ := \min\{\lambda_0^+, \, 2 c^+\}\ ,
\]
and put $\delta^+ := \lambda^+ - c^+>0$ (giving $\lambda^+ = 2c^+$ when $\lambda_0^+=\infty$). Choose once and for all some 
\begin{equation}\label{eq: epsilondef}
\varepsilon < \frac{\delta^+}{16\lambda^+},
\end{equation}
such that also
\begin{equation}\label{eq: pizzapie34}
\limsup_{\|x\|_1 \rightarrow \infty}\,\, \sup_{y: \, \|y - x\|_1 \leq \varepsilon \|x\|_1} \frac{\tau(0,y)}{\|x\|_1} < \lambda^+ - \frac{7 \delta^+}{8}\quad \mu\text{-a.s.}
\end{equation}
This follows from Lemma~\ref{l1_shape} because if $\|y\|_1$ is large, $\|y-x\|_1 \leq \e \|x\|_1$ gives $\tau(0,y)/\|x\|_1 \leq (\tau(0,y)/\|y\|_1)(1+\e) < c^+(1+\e)$. Fix $\beta > 0$ with $\mu(B(0,x_1,x_2)) > \beta.$  

The majority of the proof will consist of defining a few events in sequence, the second of which we will pull back to the space $\Omega_1$ to do the edge modification. We will need to choose further parameters to ensure that each of these events has positive probability. For an arbitrary outcome in $\widetilde{\Omega}$ and $N\geq 0$, denote by $r_0(N)$ and $r_2(N)$ the segments of $\Gamma_0$ and $\Gamma_{x_2}$ up to their first intersections with $L'_N$ (if they exist) and let $w_N$ denote the midpoint of the segment of $L'_N$ lying between these first intersections. The first event $B^\circ(R,N,\e)$ is defined by the conditions (for $R,N \geq 1$)
\begin{enumerate}
\item $\Gamma_0, \Gamma_{x_1}$ and $\Gamma_{x_2}$ never intersect,
\item they stay on the far side of $L'_0$ except for their initial vertices,
\item $\Gamma_0$ and $\Gamma_{x_2}$ intersect $L'_N$ and their first intersection points are within $\ell^1$ distance $\e N$ of each other,
\item for $i=0,2$, $\tau(r_i(N)) < (\lambda^+-7\delta^+/8)\|w_N\|_1$ and
\item $\Gamma_0$ and $\Gamma_{x_2}$ do not touch any $x \preceq L'_N$ with $x \notin S(R,N)$.
\end{enumerate}
See Figure~\ref{fig: b_circ} for a depiction of the event $B^\circ(R,N,\e)$.

\begin{figure}[h]
\caption{The event $B^\circ(R,N,\e).$ The solid dots represent the first intersection points of $\Gamma_0$ and $\Gamma_{x_2}$ with $L'_N$. They are within $\ell^1$ distance $\e N$ of each other.}
\label{fig: b_circ}
\centering
\includegraphics[scale=0.65]{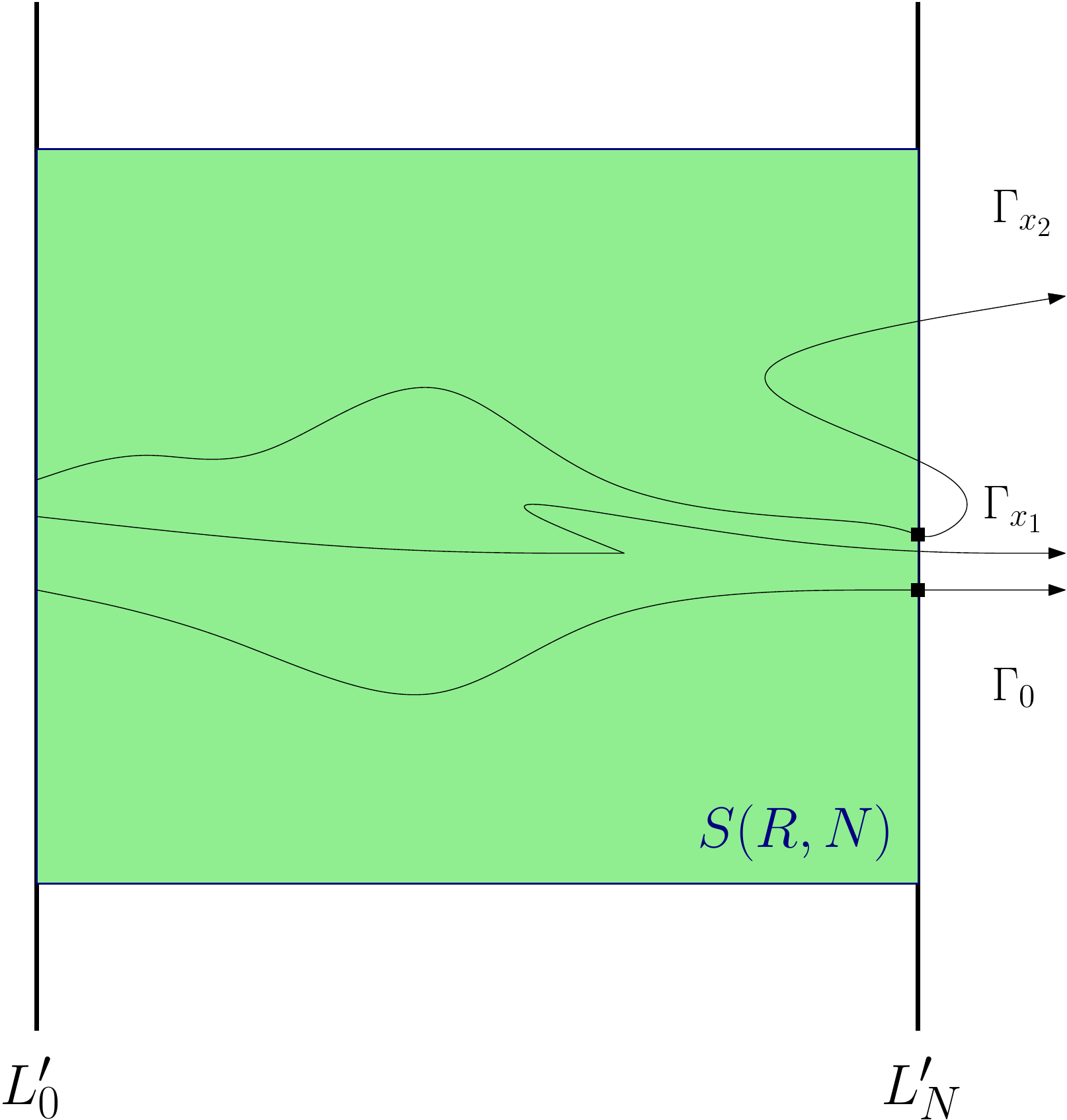}
\end{figure}

We claim that there exists $N_0$ and $R_0$ such that 
\begin{equation}\label{eq: pizzapie44}
\mu(B^\circ (R_0,N_0,\varepsilon)) > 0\ .
\end{equation}
We also need $N_0$ to satisfy a technical requirement. It will be used at the end of the proof:
\begin{equation}\label{eq: technical_condition}
\|x_2\|_1 \leq \e N_0\ .
\end{equation}

To pick $N_0$, first choose $N_1>0$ so large that if $N \geq N_1$ then
\begin{equation}\label{eq: pizzapie45}
\passage\left( \forall z,z' \text{ with } \|z\|_1 \geq N, \text{ and } \frac{\|z-z'\|_1}{\|z\|_1} \leq \e,~ \frac{\tau(0,z')}{\|z\|_1} < (\lambda^+ - \frac{7 \delta^+}{8}) \right) > 1 - \beta / 4\ ,
\end{equation}
and $\|x_2\|_1 \leq \e N$. This is possible by \eqref{eq: pizzapie34}. Write $E_0(N)$ for the event in \eqref{eq: pizzapie45} and $E_{x_2}(N)$ for $E_0(N)$ translated so that $0$ is mapped to $x_2$. Then $\mathbb{P}(B(0,x_1,x_2) \cap E_0(N) \cap E_{x_2}(N)) > \beta/2$. By Lemma~\ref{path_concentration}, we can then choose $N_0 \geq N_1$ such that
\begin{equation}\label{eq: pizzapie46}
\mu(B(0,x_1,x_2) \cap E_0(N_0) \cap E_{x_2}(N_0) \cap C(0,x_2;N_0)) > 0\ ,
\end{equation}
where $C(0,x_2;N_0)$ is the event that $\Gamma_0$ and $\Gamma_{x_2}$ intersect $L_{N_0}'$ and their first intersection points are within $\ell^1$ distance $\e N_0$ of each other. On the event in \eqref{eq: pizzapie46}, the endpoints of the $r_i(N_0)$'s are within distance $\e N_0$ of $w_{N_0}$ and since they are on $L'_{N_0}$, their $\ell^1$ distance from $0$ or $x_2$ is at least $N_0$. Therefore $\tau(r_i(N_0)) < (\lambda^+ - 7 \delta^+/8)\|w_{N_0} \|_1$ for $i=0,2$. This shows that the intersection of four of the five events in the definition of $B^\circ(R,N_0,\e)$ occurs with positive probability. For the fifth, recall that on $B(0,x_1,x_2)$, the paths $\Gamma_0$, $\Gamma_{x_1}$ and $\Gamma_{x_2}$ contain only finitely many vertices $z \preceq L'_{N_0}$. Thus we can choose $R_0$ large enough (depending on $N_0$) to satisfy condition 5 and complete the proof of \eqref{eq: pizzapie44}.

Fix these $R=R_0$ and $N=N_0$ from now on. The next event we define is a cylinder approximation of the first event. It will be needed to pull back to $\Omega_1$. For $M>0$ and $x \in \Z2,$ let $\Gamma^M_x$ be the finite path formed by starting at $x$ and then passing along out-edges of $\dirgraph$ until we first reach a vertex of $\mathbb{R}^2 \setminus (-M,M)^2$. (Note that by this definition, $\Gamma_x^M = \{x\}$ whenever $x \notin (-M,M)^2$.) We define $B^\circ_M(R,N,\varepsilon)$ with the same conditions as $B^\circ(R,N,\varepsilon),$
except replacing the paths $\Gamma_{(\cdot)}$ by the segments $\Gamma_{(\cdot)}^M$. In addition, however, we impose the restriction that, writing
\[
\partial M = [-M,M]^2 \setminus (-M,M)^2\ ,
\]
we have
\begin{equation}\label{eq: neweq}
\Gamma_y^M \cap \partial M \subseteq \{ z \in \mathbb{R}^2 : z \succ L'_N\},~ y=0,x_2\ .
\end{equation}
Of course, if $\Gamma_0^M$ (etc.) does not intersect $L'_N,$ then $B^\circ_M$ does not occur.  Then $B^\circ_M(R,N,\varepsilon)$ is open for all $M$ and we claim that
\begin{equation}\label{eq: cylinderapprox}
B^\circ (R,N,\varepsilon) = \cup_{M_0 = 1}^{\infty} \cap_{M=M_0}^{\infty} B^\circ_M(R,N,\varepsilon)\ .
\end{equation}
Assuming we show this, then there exists some $M_0$ such that $\mu(\cap_{M=M_0}^{\infty} B^\circ_M(R,N,\varepsilon)) > 0$ and so there is some $\beta'$ with
\begin{equation}\label{eq: aftercylinder}
\mu(B^\circ_M(R,N,\varepsilon)) > \beta' \text{ for all }M \geq M_0\ .
\end{equation}

To prove \eqref{eq: cylinderapprox}, note that the right side is the event that $B^\circ_M(R,N,\e)$ occurs for all $M$ bigger than some random $M_0$. Suppose that an outcome is in the left side. Then the paths $\Gamma_0$, $\Gamma_{x_1}$ and $\Gamma_{x_2}$ are disjoint and remain on the far side of $L_0'$ (except for their first vertices), so the same is true for each $\Gamma_{(\cdot)}^M$ for all $M \geq 1$. Also $\Gamma_0^M$ and $\Gamma_{x_2}^M$ do not touch any $x \preceq L_N'$ with $x \notin S(R,N)$ for all $M \geq 1$. Because $\Gamma_0$ and $\Gamma_{x_2}$ intersect $L_N'$, so do $\Gamma_0^M$ and $\Gamma_{x_2}^M$ for all $M$ bigger than some random $M_1$. Their first intersection points are the same as those of $\Gamma_0$ and $\Gamma_{x_2}$, so for $M \geq M_1$, their first intersection points with $L'_N$ are within $\ell^1$ distance $\e N$ of each other. Further, the passage times of the segments up to $L'_N$ are strictly bounded above by $(\lambda^+-7\delta^+/8)\|w_N\|_1$. Last, because $\Gamma_0$ and $\Gamma_{x_2}$ do not touch any $x \preceq L'_N$ with $x \notin S(R,N)$, they share only finitely many vertices with $\{z \in \mathbb{Z}^2 : z \preceq L'_N\}$ and so must eventually lie on the far side of $L'_N$. This allows us to further increase $M_1$ to an $M_0$ such that if $M \geq M_0$ then in addition \eqref{eq: neweq} holds.

Suppose conversely that the right side of \eqref{eq: cylinderapprox} occurs. Then for all $M$ bigger than some random $M_0$, the six events comprising $B^\circ_M(R,N,\e)$ occur. In particular, the paths $\Gamma_0$, $\Gamma_{x_1}$ and $\Gamma_{x_2}$ are disjoint and stay on the far side of $L_0'$ except for their first vertices (parts 1 and 2 of $B^\circ(R,N,\e)$). Furthermore $\Gamma_0$ and $\Gamma_{x_2}$ cannot touch any $x \preceq L_N'$ with $x \notin S(R,N)$ (part 5). For $M \geq M_0$, the paths $\Gamma_0^M$ and $\Gamma_{x_2}^M$ intersect $L_N'$, with their first intersection points within distance $\e N$ of each other (with passage time strictly bounded above by $(\lambda^+ - 7\delta^+/8)\|w\|_1$). These are the same first intersection points as $\Gamma_0$ and $\Gamma_{x_2}$, so parts 3 and 4 of $B^\circ(R,N,\e)$ occur.

We now pull the cylinder approximation $B^\circ_M(R,N,\e)$ back to $\Omega_1$ using Lemma~\ref{toalphas}. Because this is an open event and satisfies \eqref{eq: aftercylinder} for $M\geq M_0$, we can find an $M$-dependent number $K_0$ such that if $k \geq K_0$, then there is a set $\Lambda_{M,k}$ of values of $\alpha \in [0,n_k]$ which has Lebesgue measure at least $C_{\beta'} n_k$, on which $\mu_\alpha(B^\circ_M(R,N,\e)) > \beta'/2$. Pull back to $\Omega_1,$ setting $B_M^\alpha:= \Phi_\alpha^{-1}(B^\circ_M(R,N,\e)),$ where $\Phi_\alpha$ was defined in \eqref{eq: phidef}. (Here we have suppressed mention of $R,N,\e$ in the notation, as they are fixed for the remainder of the proof.) Then 
\begin{equation}\label{eq: pizzapie77}
\passage(B_M^\alpha) > \beta'/2 \text{ for all } \alpha \in \Lambda_{M,k} \text{ if } M \geq M_0 \text{ and } k \geq K_0(M)\ .
\end{equation}
We henceforth restrict to values of $M,$ $\alpha$ and $k$ such that (\ref{eq: pizzapie77}) holds. In the end of the proof we will take $k \to \infty$ and then $M \to \infty$. In particular then we will be thinking of 
\[
\alpha \gg M \gg N\ , 
\]
the latter of which is fixed. Some of the remaining definitions will only make sense for such $\alpha$, $M$ and $N$ but this does not affect the argument.

Next we define the third of our four events, now working on $\Omega_1$. Let $s^\alpha_{y}$ be the geodesic from $y \in \Z2$ to $L_\alpha$ (recall this was defined for $\varpi$ and not $\varpi'$), and $s^\alpha_y(M)$ the path $s^\alpha_y$ up to its first intersection with $\mathbb{R}^2 \setminus (-M,M)^2$. If $s_0^\alpha(M)$ and $s_{x_2}^\alpha(M)$ intersect $L_N'$ then write $r_i^\alpha(M),$ $i=0,2$ for the portions up to the first intersection point. As before, let $w^\alpha_N$ be the midpoint of the segment of $L_N'$ between these two intersection points. Let $\mathcal{R}_1^\alpha(M)$ be the closed connected subset (in $\mathbb{R}^2$) of $\{x \in \mathbb{R}^2: x \succeq L'_0\}$ with boundary curves $s_0^\alpha(M)$, $s_{x_2}^\alpha(M)$, $L_0'$ and $\partial M$. Similarly let $\mathcal{R}_2^\alpha(M)$ be the closed connected subset of $\mathcal{R}_1^\alpha(M)$ with the following boundary curves: the portions of $s_0^\alpha(M)$ and $s_{x_2}^\alpha(M)$ after their last intersections with $L_N'$, the segment of $L_N'$ between these intersections and last, $\partial M$. Note that when \eqref{eq: neweq} holds, $\mathcal{R}_2^\alpha(M)$ is contained in $\{z \in \mathbb{R}^2: z \succeq L'_N\}$. See Fig.~\ref{my_first_table} for an illustration of these definitions.

\begin{figure}[h]
\caption{The regions $\mathcal{R}_1^\alpha(M)$ and $\mathcal{R}_2^\alpha(M)$. The left figure shows $\mathcal{R}_1^\alpha(M)$ in green. It has boundary curves $L'_0$, $\partial M$, $s_0^\alpha(M)$ and $s_{x_2}^\alpha(M)$. The right figure shows $\mathcal{R}_2^\alpha(M) \subseteq \mathcal{R}_1^\alpha(M)$ in green. It has boundary curves $L'_N$, $\partial M$, and the pieces of $s_0^\alpha(M)$ and $s_{x_2}^\alpha(M)$ from their last intersections with $L'_N$. Note that $\mathcal{R}_2^\alpha(M)$ is contained in the far side of $L'_N$ by \eqref{eq: neweq}.}
\label{my_first_table}
\centering
\begin{tabular}{cc}
\includegraphics[scale=0.40]{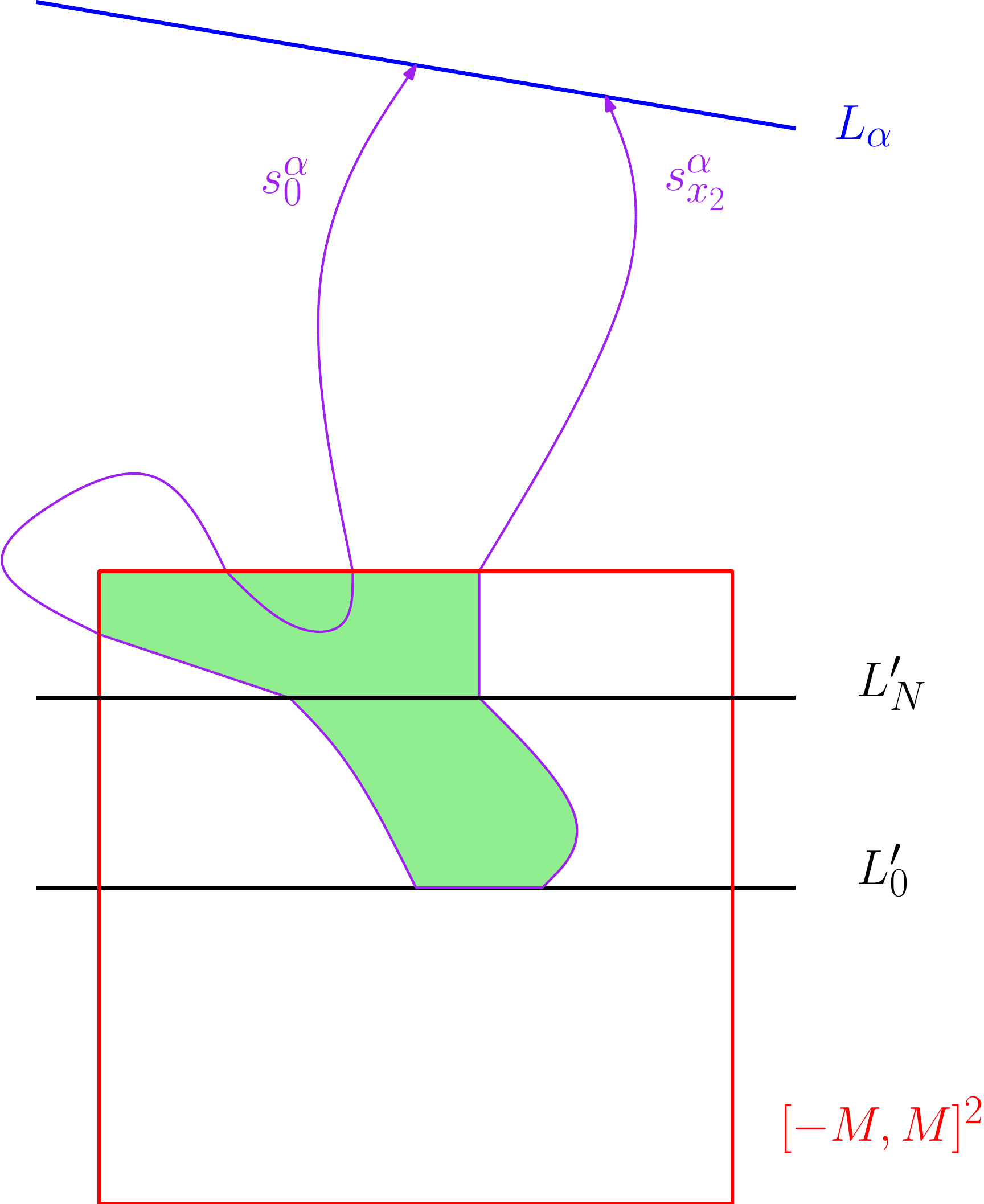}
\includegraphics[scale=0.40]{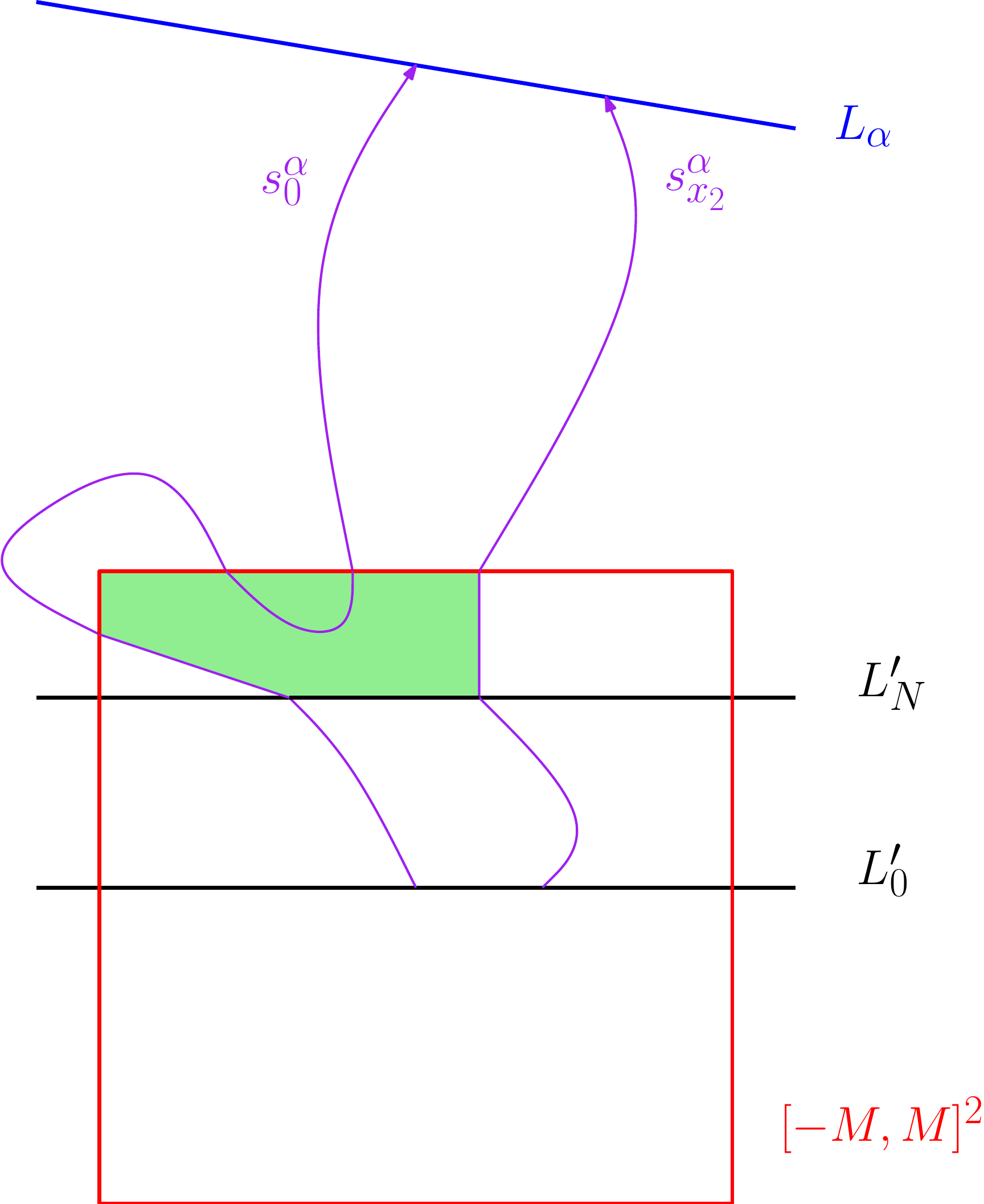}
\end{tabular}
\end{figure}

The event $\hat B_M^\alpha \subseteq \Omega_1$ is then defined by the following conditions:
\begin{itemize}
\item $s_0^\alpha(M)$ and $s_{x_2}^\alpha(M)$ intersect $L'_0$ only once, are disjoint, and do not touch any $y \preceq L'_N$ with $y \notin S(R,N)$.
\item $s_0^\alpha(M)$ and $s_{x_2}^\alpha(M)$ intersect $L'_N$ and their first intersection points are within $\ell^1$ distance $\varepsilon N$ of each other; the paths $r_i^\alpha(M)$ satisfy $\tau(r_i^\alpha(M)) < (\lambda^+ -7 \delta^+/8) \|w^\alpha_N\|_1,$ for $i=0,2$.
\item $s_y^\alpha(M) \cap \partial M \subseteq \{z \in \mathbb{R}^2 : z \succ L'_N\}$ for $y=0,x_2$,
\item there is a vertex $X^* \in L'_N \cap S(R,N)$ such that $s_{X^*}^\alpha(M)$ is disjoint from $s_0^\alpha(M)$ and $s_{x_2}^\alpha(M)$ but is contained in $\mathcal{R}_2^\alpha(M)$, and 
\item the portions of $s_0^\alpha,$ $s_{X^*}^\alpha$ and $s_{x_2}^\alpha$ beyond $[-M,M]^2$ do not contain a vertex of $S(R,N)$;
\end{itemize}

We claim there is an $M_0' \geq M_0$ such that
\begin{equation}\label{eq: pizzapie99}
\mathbb{P}(\hat B_M^\alpha) > \beta'/4 \text{ for all } M \geq M_0'\ .
\end{equation}
Verifying this requires us to define an auxiliary event. Let $H_{M} \subseteq \Omega_1$ denote the event that no geodesic from any point in $S(R,N)$ returns to $S(R,N)$ after its first intersection with $\partial M.$  Then $\passage(H_M) \rightarrow 1$ as $M \rightarrow \infty.$  So for any $M$ larger than some $M_0' \geq M_0$, $\mathbb{P}(H_M) > 1-\beta'/4$, giving
\[ 
\passage(B_M^\alpha \cap H_M) > \beta'/4 \text{ for all } M \geq M_0'\ .
\]
To finish the proof of \eqref{eq: pizzapie99} we show that $B_M^\alpha \cap H_M \subseteq \hat B_M^\alpha$. Note that the first three conditions of $\hat B_M^\alpha$ are immediately implied by $B_M^\alpha$; they are the analogues on $\Omega_1$ of the conditions that make up $B_M^\circ(N,R,\e)$ (each $\Gamma_{(\cdot)}^M$ is replaced by $s_{(\cdot)}^\alpha(M)$). For the fourth condition, note that when $B_M^\alpha$ occurs, $s_0^\alpha(M)$, $s_{x_1}^\alpha(M)$ and $s_{x_2}^\alpha(M)$ stay on the far side of $L_0'$ (aside from their initial vertices) and stop when they touch $\partial M$. Therefore by planarity, $s_{x_1}^\alpha(M)$ is contained in $\mathcal{R}_1^\alpha(M)$. In particular, if we choose $X^*$ to be the last intersection point of $s_{x_1}^\alpha(M)$ with $L'_N$, then $s_{X^*}^\alpha(M)$ is trapped in $\mathcal{R}_2^\alpha(M)$. We can see this as follows. The last vertex of $s_{X^*}^\alpha(M)$ is clearly in this region because it must be in $\mathcal{R}_1^\alpha(M) \cap \partial M$ and this equals $\mathcal{R}_2^\alpha(M) \cap \partial M$. Proceeding backward along $s_{X^*}^\alpha(M)$ from this final vertex, the path can only leave $\mathcal{R}_2^\alpha(M)$ if it (a) leaves $[-M,M]^2$ (b) crosses $s_0^\alpha(M)$ or $s_{x_2}^\alpha(M)$ or (c) crosses $L'_N$. Because none of these can happen, the fourth condition holds. As for the fifth, it is implied by $H_M$, so we have proved \eqref{eq: pizzapie99}.

Our fourth and final event will fix some random objects to be deterministic so that we can apply the edge modification lemma. On the event $\hat B_M^\alpha$, let $U$ denote the (random) closed connected subset of $[-M,M]^2$ with boundary curves $L'_0$, $L'_N$, $r_0^\alpha(M)$ and $r_2^\alpha(M)$. Note that $U \subseteq S(R,N)$. Furthermore we note that on $\hat B_M^\alpha$, $U \cap \mathcal{R}_2^\alpha(M)$ is contained in $L_N'$. This is because $\mathcal{R}_2^\alpha(M) \subseteq \{z : z \succeq L'_N\}$, whereas $U\subseteq \{z : z \preceq L'_N\}$. Last, define $U_\mathcal{E}$ to be the random set of edges with both endpoints in $U$ and which are not edges in $s_0^\alpha(M),s_{x_2}^\alpha(M), L_0'$ or $L_N'$. See Figure~\ref{fig: bddcase} for an illustration of these definitions. 

\begin{figure}[h]
\caption{Illustration of definitions on $\hat B_M^\alpha$. The region $U$ is in blue and is contained in $S(R,N)$ (not pictured). It is bounded by curves $L'_0$, $L'_N$, $r_0^\alpha(M)$ and $r_2^\alpha(M)$. The path $s_{x^*}$ begins at the final intersection point of the dotted path with $L'_N$.}
\label{fig: bddcase}
\centering
\includegraphics[scale=0.65]{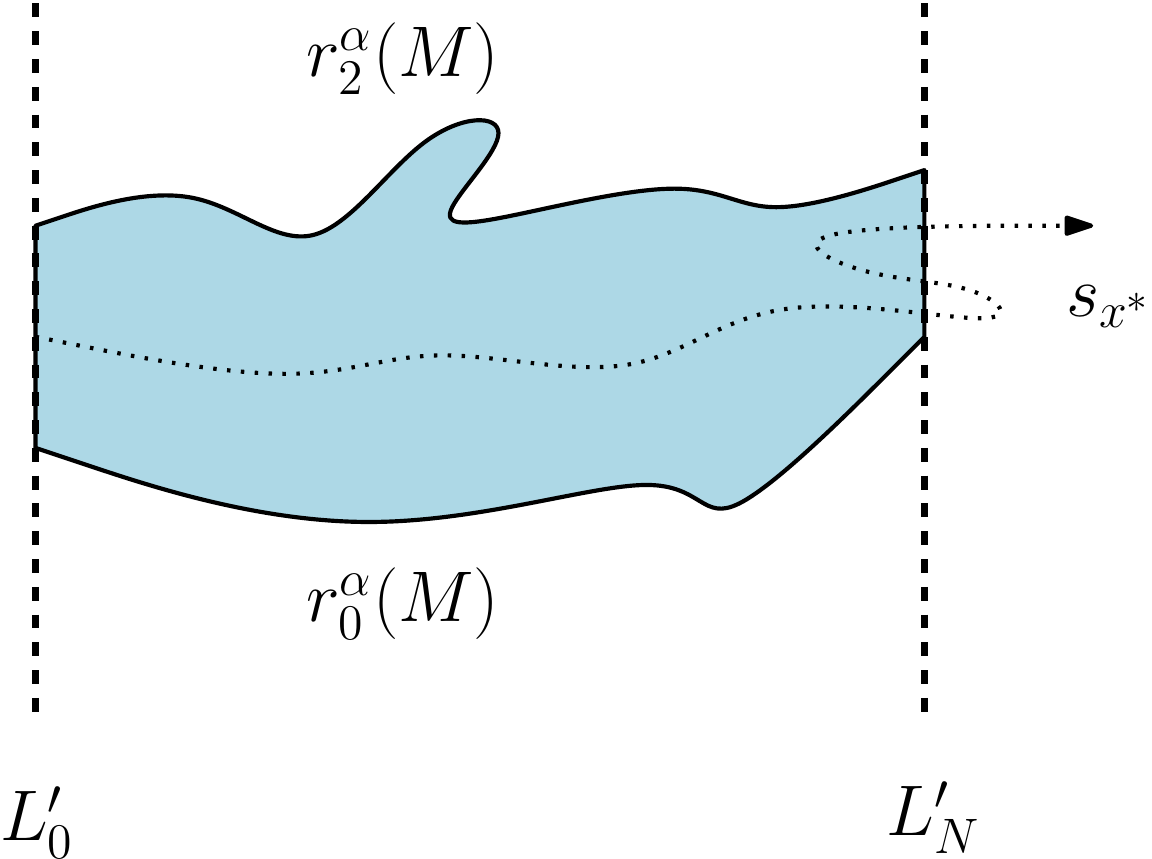}
\end{figure}

On $\hat{B}_M^\alpha,$ there are at most $2^{64NR}$ possibilities for $U$ and $U_\mathcal{E}$ and at most $2R$ choices for $X^*.$ So there exist some deterministic $U',$ $U_{\mathcal{E}}',$ and $x^*$ such that, if we define
\[
\tilde{B}_M^\alpha:= \hat{B}^\alpha_M \cap \{U = U', \, U_\mathcal{E} = U'_{\mathcal{E}}\}\cap \{X^* = x^*\}\ , 
\]
then
\begin{equation}\label{eq: nachos_bellgrande}
\passage(\tilde{B}_M^\alpha) > 2^{-2 - 64NR} \beta' / 2R \text{ for } M \geq M_0' \text{ and }\alpha \in \Lambda_{M,k}\ . 
\end{equation}
The meaning of the event $\{X^*=x^*\}$ is that the deterministic point $x^*$ satisfies the conditions in the fourth and fifth items of the description of $\hat B_M^\alpha$.

In the rest of the proof we perform the edge modification and push forward to $\widetilde \Omega$. To apply Lemma~\ref{lem: edge_modification} we need to verify that $\tilde B_M^\alpha$ is $e$-increasing for all $e \in U'_{\mathcal{E}}$. For this purpose, suppose that $\omega \in \tilde B_M^\alpha$ and that $\omega'$ is another configuration such that $\omega_e' \geq \omega_e$ for some fixed $e \in U'_{\mathcal{E}}$ but $\omega_f' = \omega_f$ for all other $f \neq e$. By construction, $e$ is not an edge of $s_0^\alpha(M)$, $s_{x^*}^\alpha(M)$ or $s_{x_2}^\alpha(M)$ ($e \notin s_{x^*}^\alpha(M)$ since $e$ is contained in $U_\mathcal{E}$, which does not meet $L_N'$, so is not in $\mathcal{R}_2^\alpha(M) \supseteq s_{x^*}^\alpha(M)$). Furthermore because $s_0^\alpha$, $s_{x^*}^\alpha$ and $s_{x_2}^\alpha$ do not re-enter $S(R,N)$ after leaving $[-M,M]^2$ and all edges of $U'_{\mathcal{E}}$ have both endpoints in $S(R,N)$, $e$ cannot be on these paths either. This means that 
\[
s_y^\alpha(\omega) = s_y^\alpha(\omega') \text{ for } y=0,x^*,x_2 \text{ and } U(\omega) = U(\omega'),~U_{\mathcal{E}}(\omega) = U_{\mathcal{E}}(\omega')\ .
\]
So the fifth condition of $\hat B_M^\alpha$ occurs in $\omega'$. The paths $s_y^\alpha(M)$ are then equal in $\omega$ and $\omega'$, so conditions 1, the first part of 2, and 3 and 4 hold in $\omega'$. As $e$ is not on any of these paths, their passage times are the same in $\omega'$. This gives the second part of condition 2 of $\hat B_M^\alpha$ and shows that $\tilde B_M^\alpha$ is $e$-increasing.

Now we conclude the proof in a slightly different manner depending on whether or not $\lambda_0^+$ is finite; we focus first on the case that $\lambda_0^+ < \infty.$ We will use Lemma~6.6, but several times in sequence, appending events onto $\hat B_M^\alpha$. Precisely we note for reference that if $e_1, \ldots, e_j$ are edges and $a_1, \ldots, a_j \in \mathbb{R}$ then
\[
\hat B_\alpha^M \cap \left[ \cap_{i=1}^j \{ \omega_{e_i} \geq a_i \} \right] \text{ is } e\text{-increasing for } e \in U'_{\mathcal{E}}\ .
\]
Using Lemma~\ref{lem: edge_modification} once for each edge $e \in U'_{\mathcal{E}}$ and the upper bound $|U'_{\mathcal{E}}| \leq 32 N R$, we can find some constant $C_{N,R}$ such that, defining
\[ 
B_M'^\alpha := \tilde{B}_M^\alpha \cap \left\{ \forall e \in U'_{\mathcal{E}}, \, \omega_e \geq \lambda^+ - \delta^+ / 4 \right\}\ ,
\]
we have
\[ 
\passage \left( B_M'^\alpha \right) > C_{N,R} > 0 \text{ for all }M \geq M_0' \text{ and } \alpha \in \Lambda_{M,k} \text{ when } k \geq K_0(M)\ .
\]
(For the first application of the lemma we use $\vartheta = 2^{-2-64NR}\beta'/2R$, for the second, a smaller $\vartheta$, and so on.)

We claim that on $B_M'^\alpha,$ no $z \in \Z2 \cap [-M,M]^2$ with $z \preceq L'_N$ and $z \notin S(R,N)$ has $s_z^\alpha(M) \cap s_{x^*}^\alpha(M) \neq \varnothing$. We argue by first estimating the passage time between vertices from $L_0'$ to $L_N'$ in $U'$. For any outcome in $B_M'^\alpha,$ given vertices $x \in U' \cap L'_0$ and $y \in U' \cap L'_N,$ there is a path from $x$ to $y$ formed by moving along $L'_0$ to $0,$ taking $r_0^\alpha$ to $L'_N,$ and moving similarly along $L'_N$ to $y.$  This gives
\begin{equation}
\label{fastish_path}
\tau(x,y) < (\lambda^+ - 7\delta^+/8)\|w_N^\alpha\|_1+ (N\varepsilon + \|x_2\|_1)\lambda^+ .
\end{equation}
Using the choice of $\varepsilon$ from \eqref{eq: epsilondef} and condition \eqref{eq: technical_condition} to bound the right side of (\ref{fastish_path}),
\begin{equation}
\label{whoa_so_fast}
\tau(x,y) \leq (\lambda^+ - 3 \delta^+ / 4) \|w_N^\alpha\|_1.
\end{equation}
Suppose now that a point $z$ exists as in the claim. Since $s_0^\alpha(M)$ and $s_{x_2}^\alpha(M)$ do not touch any $y \notin S(R,N)$ with $y \preceq L'_N$ (see item 1 in the definition of $\hat B_M^\alpha$), 
\[
\mathcal{R}_1^\alpha(M) \cap \{y  : y \preceq L'_N\} \subseteq S(R,N)\ .
\]
This implies $z \notin \mathcal{R}_1^\alpha(M)$, whereas $x^* \in \mathcal{R}_1^\alpha(M)$. As $s_z^\alpha(M)$ cannot touch $s_0^\alpha(M)$ or $s_{x_2}^\alpha(M)$ (else it would merge with one of them) it would have to enter $\mathcal{R}_1^\alpha(M)$ through $L_0'$ and pass through all of $U'$ from $L'_0$ to $L'_N$, thus taking only edges of $U'_{\mathcal{E}}.$  The portion $\gamma'$ of $\gamma$ from its first intersection with $L_0'$ to its first intersection with $L_N'$ would then satisfy
\begin{align*}
\tau(\gamma') &\geq \left(\lambda^+ - \delta^+ / 4 \right) \left[ \|w^\alpha_N\|_1 - \|x_2\|_1 - N \varepsilon\right]\\
	&\geq (\lambda^+ - \delta^+ / 4) \|w^\alpha_N\|_1 - 2 \|w^\alpha_N\|_1 \varepsilon \lambda^+\\
	&\geq (\lambda^+  - 3 \delta^+ / 8) \|w^\alpha_N\|_1,
\end{align*}
in contradiction with the estimate of (\ref{whoa_so_fast}). This establishes the claim.

For the final step in the case that $\lambda_0^+<\infty$, note that by the previous claim, the pushforward, $\Phi_\alpha (B_M'^\alpha)$, is a sub-event of $B_M'= B_M'(x^*;R,N),$ defined exactly as the event $B'=B'(x^*;R,N)$ in \eqref{eq: B_prime_def} except with $\Gamma_{x^*}$ and $\Gamma_z$ replaced by the truncated paths $\Gamma_{x^*}^M$ and $\Gamma_z^M$ and considering only $z \in [-M,M]^2$. Thus 
\[
\mu_\alpha(B_M') \geq C_{N,R} \text{ for all } M \geq M_0',~ k \geq K_0(M) \text{ and } \alpha \in \Lambda_{M,k}\ ,
\]
with $\Lambda_{M,k} \subseteq [0,n_k]$ of Lebesgue measure at least $C_{\beta'}n_k$. As the indicator of $B_M'$ is continuous,
\[
\mu(B_M') = \lim_{k \to \infty} \mu_{n_k}^*(B_M') \geq C_{N,R} C_{\beta'}\ .
\]
Last, 
\[
\mu(B') = \mu(B_M' \text{ for infinitely many }M) \geq C_{N,R} C_{\beta'} > 0\ ,
\]
completing the proof in the case $\lambda_0^+ < \infty$.

If $\lambda_0^+ = \infty,$ we are no longer guaranteed the estimate (\ref{whoa_so_fast}), since the passage time of a path taking $N \varepsilon$ steps along $L'_N$ is not necessarily bounded above by $N \varepsilon \lambda^+.$ However, writing $\tilde E$ for the set of edges with an endpoint within $\ell^1$ distance 1 of $U'$ but not in $U'_{\mathcal{E}}$ and noting
\[
A_C := \{\text{for all } e \in \tilde E,~ \tau_e \leq C\}
\]
satisfies $\mathbb{P}(A_C) \to 1$ as $C \to \infty$ independently of $k$ and $M$, we can choose $C_{\text{big}}$ such that 
\[
\mathbb{P}(\tilde B_M^\alpha \cap A_{C_{\text{big}}}) > 0
\]
independently of $k$ and $M$. This event is still monotone increasing in the appropriate edge variables. In particular, we can modify the edges in $U'_{\mathcal{E}}$ to be each larger than $2C_{\text{big}} |\tilde E|$ and the rest of the proof follows as in the case $\lambda_0^+< \infty$.
\end{proof}

\subsection{Deriving a contradiction}\label{sec: contradiction}

Given that the event $B'(x^*;R,N)$ of the preceding section has positive probability, we now derive a contradiction, proving that all paths in $\mathbb{G}$ must merge. The next lemma is an example of a mass-transport principle. (See \cite{BLS, Haggstrom1, Haggstrom2} for a more comprehensive treatment.)

\begin{lem}
\label{mass_transport}
Let $m: \Z2 \times \Z2 \rightarrow [0,\infty)$ be such that $m(x,y) = m(x+z,y+z)$ for all $x,y,z \in \mathbb{Z}^2.$
Then
\[
\forall x \in \Z2, \quad \sum_{y \in \Z2} m(x,y) = \sum_{y \in \Z2} m(y,x)\ .
\]
\end{lem}
\begin{proof}
Write
\begin{align*}
\sum_{y \in \Z2} m(x,y) = \sum_{z \in \Z2} m(x,x+z) = \sum_{z \in \Z2} m(x-z,x) = \sum_{y \in \Z2} m(y,x)\ .
\end{align*}
\end{proof}
Given a realization of $\mathbb{G}$ and $x \in \mathbb{Z}^2$, order the set 
\begin{equation}\label{eq: backclusterdef}
C_x = \{ y \in \Z2: y \to x \text{ in } \mathbb{G}\}
\end{equation}
using a dictionary-type ordering where $y$ precedes $y'$ if either $\varpi' \cdot y < \varpi' \cdot y'$ or if both $\varpi' \cdot y = \varpi' \cdot y'$ and $y \cdot \varsigma < y' \cdot \varsigma$ (where $\varsigma$ was fixed in \eqref{eq: chicken_alfredo}); clearly this defines a total ordering. If there is a least element $y$ under this ordering, we will call $y$ the progenitor of $x$ (relative to $\mathbb{G}$).
We define the $\mathbb{G}$-dependent function $m_\mathbb{G}$ on pairs of vertices $x,y$ by
\[ m_{\mathbb{G}}(x,y) = \begin{cases}
	1 &\text{if $y$ is the progenitor of $x$}\\
	0 &\text{otherwise},
	\end{cases} \]
and let $m(x,y):= \mathbb{E}_\mu(m_{\mathbb{G}}(x,y)).$ Note that $m(x,y) = m(x+z,y+z)$ by the fact that $\mathbb{G}$ has a translation-invariant distribution.

Since each $x$ can have at most one progenitor, 
\begin{equation}
\label{small_mass}
\sum_{y\in \mathbb{Z}^2} m(x,y) \leq 1 \text{ for all } x\in \mathbb{Z}^2\ .
\end{equation}
On the other hand, if $B'(x^*; R,N)$ occurs, then $\Gamma_z$ cannot intersect $\Gamma_{x^*}$ if $z \preceq L'_N$ and $z \notin S(R,N).$ Therefore, on this event, there is some vertex $y \in S(R,N)$ which is the progenitor of infinitely many vertices of $\Gamma_{x^*}.$ In particular,
\begin{equation}
\label{big_mass}
\sum_{y \in \mathbb{Z}^2} m(y,x) = \infty.
\end{equation}
The contradiction implied by (\ref{small_mass}), (\ref{big_mass}) and Lemma \ref{mass_transport} gives $\mu(B'(x^*;R,N))=0$. However this contradicts the previous section and completes the proof of Theorem~\ref{thm: Gcoalescethm}.

\subsection{Absence of backward infinite paths}

In this section, we move on from Theorem~\ref{thm: Gcoalescethm} to show that because all paths in $\mathbb{G}$ coalesce, all paths in the ``reverse" direction terminate. That is, recalling the definition of $C_x$ in \eqref{eq: backclusterdef},

\begin{thm}
\label{no_back_path}
For each $x \in \Z2,$ $|C_x| < \infty$ with $\mu$-probability one.
\end{thm}

\begin{rem}
The proof below applies to the following general setting. Suppose $\nu$ is a translation-invariant probability measure on directed subgraphs of $\mathbb{Z}^2$ and there is a line $L \subseteq \mathbb{R}^2$ such that $\nu$-almost surely (a) each $x$ has exactly one forward path and it is infinite (b) all forward paths coalesce and (c) each forward infinite path emanating from a vertex on $L$ intersects it finitely often. Then all backward clusters are finite $\nu$-almost surely.
\end{rem}

We assume that, contrary to the theorem, there exists $x \in \Z2$ with
$\mu(|C_x|=\infty)>0$ for the remainder of this section to derive a contradiction. Using Lemma~\ref{lem: varpilemma}, choose a deterministic $\varpi'$ with argument in $\{j \pi/4 : j = 0, \ldots, 7\}$ such that with positive $\mu$-probability on $\{|C_x| = \infty\}$, each $\Gamma_z$ eventually lies on the far side of each $L'_N$. Note that this event is translation-invariant, so by conditioning on it, we may assume that it occurs with probability 1 (and $\mu$ is still translation-invariant).

\begin{clam}
There exist vertices $z \neq z'$ in $L'_0$ such that
\begin{equation}
\label{triple_pt}
 \mu\left(|C_z| = \infty,\, |C_{z'}| = \infty, \, \Gamma_z \cap L'_0 = \{z\}, \, \Gamma_{z'} \cap L'_0 = \{z'\}\right) > 0\ .
\end{equation}
\end{clam}
\begin{proof}
By translation-invariance, we may assume that the $x$ with $\mu(|C_x|=\infty)>0$ satisfies $x \prec L'_0.$ $\mu$-almost surely, $\Gamma_x$ has a last intersection with $L'_0.$ There are countably many choices for such a last intersection, so there exists a vertex $z \in L'_0$ such that
\[
\mu\left( |C_z| = \infty, \, \Gamma_z \cap L'_0 = \{z\}\right) > 0\ .
\]
Translating by $\varsigma$ (chosen from \eqref{eq: chicken_alfredo}), the ergodic theorem gives $z, z'$ satisfying (\ref{triple_pt}).
\end{proof}

\begin{proof}[Proof of Theorem \ref{no_back_path}.]
Given an outcome in the event in (\ref{triple_pt}), $\Gamma_{z}$ and $\Gamma_{z'}$ almost surely merge. So there is some random $z_{\mathbb{G}} \in \Z2$ which is the first intersection point of $\Gamma_{z}$ and $\Gamma_{z'}$ (``first" in the sense of both the ordering in $\Gamma_z$ and in the ordering of $\Gamma_{z'}$). Again $z_{\mathbb{G}}$ can take only countably many values, and so there is a $z_0$ which occurs with positive probability; call the intersection of the event in (\ref{triple_pt}) with the event $\{z_{\mathbb{G}} = z_0\}$ by the name $B.$

We now consider the graph $\mathbb{G}$ as an undirected graph, in which vertices $x$ and $y$ are adjacent if $\langle x,y \rangle$ or $\langle y,x \rangle$ are in $\mathbb{G}$ (we abuse notation by using the same symbol for both the directed and undirected versions of $\mathbb{G}$). We define an encounter point of the undirected $\mathbb{G}$ to be a vertex whose removal splits $\mathbb{G}$ into at least three infinite components. Note that $B \subseteq \{z_0$ is an encounter point$\}$; by translation invariance, we see that there is a uniform $c_t > 0$ such that the probability of any fixed vertex to be an encounter point is at least $c_t.$

We are now in the setting of Burton-Keane \cite{burtonkeane}. To briefly synopsize, the number of points on the boundary of $[-M,M]^2$ must be at least the number of encounter points within. In particular, the number of encounter points is surely bounded above by $8M$. But since each point within has probability at least $c_t$ to be an encounter point, the expected number of encounter points within $[-M,M]^2$ is at least $c_t M^2.$ This is a contradiction for large $M.$

\end{proof}

\section{Proofs of main theorems}\label{sec: proofs}

\subsection{Proof of Theorem~\ref{thm: sectors}}\label{sec: sectors}
Suppose that $\partial \mathcal{B}$ is differentiable at $v_\theta = \varpi$ and construct the measure $\mu$ as in Section~\ref{sec: mudef}. Using the notation of Theorem~\ref{thm: nachostheorem}, we set 
\[
L_\varrho = \{x \in \mathbb{R}^2 : x \cdot \varrho = 1\}\ .
\]
From the theorem, we deduce that with $\mu$-probability 1, $\Gamma_0$ is asymptotically directed in $J_\varrho$. But by the assumption of differentiability, $J_\varrho = I_\theta$ with $\mu$-probability 1 and thus
\begin{equation}\label{eq: ptonboyz1}
\mu\left( \Gamma_0 \text{ is asymptotically directed in } I_\theta \right) = 1\ .
\end{equation}
By Proposition~\ref{prop: firstGG2}, each finite piece of $\Gamma_0$ is a geodesic, so $\Gamma_0$ is an infinite geodesic. Define $\hat \Omega \subseteq \Omega_1$ as the set
\[
\hat \Omega = \{\omega \in \Omega_1 : \mu(\Gamma_0 \text{ is asymptotically directed in }I_\theta \mid \omega) = 1\}\ .
\]
The inner probability measure is the regular conditional probability measure. The set $\hat \Omega$ is measurable and because the marginal of $\mu$ on $\Omega_1$ is $\mathbb{P}$, it satisfies $\mathbb{P}(\hat \Omega)=1$. Further, for each $\omega \in \hat \Omega$ there is an infinite geodesic from 0 which is asymptotically directed in $I_\theta$.

\subsection{Proof of Theorem~\ref{thm: newman}}\label{sec: newman}

In this section we assume either {\bf A1'} or {\bf A2'}. Assume that the limit shape $\mathcal{B}$ has uniformly positive curvature. Then the boundary $\partial \mathcal{B}$ cannot contain any straight line segments. This implies that the extreme points $ext(\mathcal{B})$ are dense in $\partial \mathcal{B}$. Choose some countable set $D \subseteq ext(\mathcal{B})$ that is dense in $\partial \mathcal{B}$. For any $\theta_1$ and $\theta_2$ with $0<dist(\theta_1,\theta_2) < \pi$, let $I(\theta_1,\theta_2)$ be the set of angles corresponding to the shorter closed arc of $\partial \mathcal{B}$ from $v_{\theta_1}$ to $v_{\theta_2}$. By Corollary~\ref{cor: extreme}, for each $\theta_1,\theta_2 \in D$ with $0<dist(\theta_1,\theta_2)<\pi$, with probability one there is an infinite geodesic from 0 asymptotically directed in $I(\theta_1,\theta_2)$. The collection of such sets of angles is countable, so there exists an event $\Omega'\subseteq \Omega_1$ such that $\mathbb{P}(\Omega') = 1$ and for each $\omega \in \Omega'$,
\begin{enumerate}
\item for each $\theta_1, \theta_2 \in D$ such that $dist(\theta_1,\theta_2) < \pi$, there exists an infinite geodesic containing 0 and asymptotically directed in $I(\theta_1,\theta_2)$ and
\item for each $x,y \in \mathbb{Z}^2$ there is exactly one geodesic from $x$ to $y$.
\end{enumerate}
We claim that for each $\omega \in \Omega'$, both statements of the theorem hold: for each $\theta$ there is an infinite geodesic with asymptotic direction $\theta$ and each infinite geodesic has a direction. 

To prove the first statement, let $\omega \in \Omega'$ and $\theta \in [0,2\pi)$. For distinct angles $\theta_1$ and $\theta_2$ such that $0 < dist(\theta_i,\theta)< \pi$ we write $\theta_1 >_\theta \theta_2$ if $I(\theta_1,\theta)$ contains $\theta_2$. Because $D$ is dense in $\partial \mathcal{B}$, we can find two sequences $(\theta_n^1)$ and $(\theta_n^2)$ such that (a) $0 < dist(\theta_n^i,\theta)<\pi$ for all $n$ and $i$, (b) for $i=1,2$, $dist(\theta_n^i,\theta) \to 0$ as $n \to \infty$ and (c) for each $i=1,2$ and $n$, $\theta_n^j >_\theta \theta_{n+1}^j$. Let $v_n$ be the point $nv_\theta$ and let $\gamma_n$ be the geodesic from $0$ to $v_n$. Define $\gamma$ as any subsequential limit of $(\gamma_n)$. By this we mean a path $\gamma$ such that for each finite subset $E$ of $\mathbb{R}^2$, the intersection $\gamma_n \cap E$ equals $\gamma \cap E$ for all large $n$. We claim that $\gamma$ has asymptotic direction $\theta$.

Let $\e>0$ and choose $N$ such that $dist(\theta,\theta_N^j)<\e$ for $j=1,2$. Because $\omega \in \Omega'$, for $j=1,2$, we can choose an infinite geodesic $\gamma_N^j$ containing 0 with asymptotic direction in $I(\theta_N^j,\theta_{N+1}^j)$. Write $P$ for the union of $\gamma_N^1$ and $\gamma_N^2$. This complement of $P$ in $\mathbb{R}^2$ consists of two open connected components (as $P$ cannot contain a circuit). Because both paths are directed away from $\theta$, exactly one of these two components contains all but finitely many of the $nv_\theta$'s. Let $C_1$ be the union of $P$ with this component and let $C_2$ be the other component.

Choose $N_0$ so that $nv_\theta \in C_1$ for all $n \geq N_0$. We claim now that each finite geodesic $\gamma_n$ for $n \geq N_0$ is contained entirely in $C_1$. If this were not true, $\gamma_n$ would contain a vertex $z$ in $C_2$ and therefore it would cross $P$ to get from $z$ to $v_n$. Then if $w$ is any vertex on $\gamma_n \cap P$ visited by $\gamma_n$ after $z$, then there would be two different geodesics from $0$ to $w$ and this would contradict unique passage times. Therefore, as $\gamma_n$ is contained in $C_1$ for all large $n$, so must $\gamma$. This implies that $\gamma$ is asymptotically directed in the set of angles within distance $\e$ of $\theta$ (for each $\e>0$) and therefore has asymptotic direction $\theta$.

To prove the second statement choose $\omega \in \Omega'$ and let $\gamma$ be an infinite geodesic. If $\gamma$ does not have an asymptotic direction then, writing $x_n$ for the $n$-th vertex of $\gamma$, we can find an angle $\phi \in [0,2\pi)$ such that $\phi$ is a limit point of $\{\arg x_n : n \geq 1\}$ (under the metric $dist$) but $(\arg x_n)$ does not converge to $\phi$. So there exists a number $\e$ with $0<\e<\pi$ and a subsequence $(x_{n_k})$ of $(x_n)$ such that for each $m$, $dist(\arg x_{n_{2m}}, \phi) < \e/2$ but $dist(\arg x_{n_{2m+1}}, \phi) > \e$. By the first part of the theorem we can find infinite geodesics $\gamma_1$ and $\gamma_2$ from $0$ such that $\gamma_1$ has asymptotic direction $\phi + 3\e/4$ and $\gamma_2$ has asymptotic direction $\phi - 3\e/4$. Now it is clear that if we write $P$ for the union of $\gamma_1$ and $\gamma_2$ then $\gamma$ must both contain infinitely many vertices of $P$ and infinitely many vertices of $P^c$. This again contradicts unique passage times.

\begin{proof}[Proof of Corollary~\ref{cor: newman2}]
If $\theta$ is an exposed point of differentiability then by Corollary~\ref{cor: exposed}, with probability one there exists an infinite geodesic from 0 in each rational direction. Then the proof above goes through with minor modifications.
\end{proof}

\subsection{Proof of Theorem~\ref{thm: random_hyperplanes}}

Assume either {\bf A1'} or both {\bf A2'} and the upward finite energy property. Let $v \in \mathbb{R}^2$ be nonzero and $\e>0$. We will prove that the statement of the theorem holds with probability at least $1-\e$. Choose $\varpi \in \partial \mathcal{B}$ to be parallel to $v$ and construct a measure $\mu$ as in Section~\ref{sec: mudef}. Let $(n_k)$ be an increasing sequence such that $\mu_{n_k}^* \to \mu$ weakly.

We will define a double sequence of cylinder events that approximate the events in the theorem. For $m \leq n$, a configuration $\eta \in \Omega_3$ and $x,y \in [-m,m]^2 \cap \mathbb{Z}^2$, we say that $x$ is $n$-connected to $y$ ($x \to_n y$) if there exists a directed path from $x$ to $y$ whose vertices stay in $[-n,n]^2$. We say that $x$ and $y$ are $n$-connected $(x \leftrightarrow_n y)$ if there is an undirected path connecting $x$ and $y$ in $[-n,n]^2$. For $m \leq n$ write $A_{m,n} \subseteq \Omega_3$ for the event that
\begin{enumerate}
\item all vertices $v \in [-m,m]^2$ have exactly one forward neighbor in $\mathbb{G} \cap [-n,n]^2$,
\item there is no undirected circuit contained in $[-m,m]^2$,
\item for all vertices $v,w \in [-m,m]^2$, there exists $z \in [-n,n]^2$ such that $v \to_n z$ and $w \to_n z$ and
\item for all vertices $v \in [-m,m]^2$ there is no $z \in [-n,n]^2 \setminus (-n,n)^2$ such that $z \to_n v$.
\end{enumerate}

We claim that for any $m$ there exists $n(m) \geq m$ such that $\mu(A_{m,n(m)}) > 1-\e/4^{m+2}$. To prove this, let $\hat \Omega \subseteq \widetilde \Omega$ be the event that (a) all vertices have one forward neighbor in $\mathbb{G}$, (b) $\mathbb{G}$ has no undirected circuits, (c) for all $x,y \in \mathbb{Z}^2$, $\Gamma_x$ and $\Gamma_y$ coalesce and (d) $|C_x| < \infty$ for all $x \in \mathbb{Z}^2$. By Proposition~\ref{prop: secondGG2}, Theorem~\ref{thm: Gcoalescethm} and Theorem~\ref{no_back_path}, the $\mu$-probability of $\hat \Omega$ is 1. Therefore conditions 1 and 2 above have probability 1 for all $m$ and $n$. For any configuration in $\hat \Omega$ and $m \geq 1$ we can then choose a random and finite $N(m) \geq m$ to be minimal so that conditions 3 and 4 hold for all $n \geq N(m)$. Taking $n(m)$ so large that $\mu(N(m) \geq n(m)) \leq \e/4^{m+1}$ completes the proof of the claim.

We now pull $A_{m,n(m)}$ back to $\Omega_1$, using the fact that it is a cylinder event in $\Omega_3$ and thus its indicator function is continuous. There is an $m$-dependent number $K_0(m)$ such that if $k \geq K_0(m)$ then $\mu_{n_k}^*(A_{m,n(m)}) > 1-\e/4^{m+2}$. By definition of $\mu_{n_k}^*$ in \eqref{eq: munstar} and $\Phi_\alpha$ in \eqref{eq: phidef}, the set $\Lambda_{m,k}$ of values of $\alpha \in [0,n_k]$ such that $\mathbb{P}(\Phi_\alpha^{-1}(A_{m,n(m)})) > 1-\e/2^{m+2}$ has Lebesgue measure at least $n_k(1-2^{-(m+2)})$. 

The next step is to construct a deterministic sequence $(a_m)_{m \geq 1}$ of real numbers such that
\begin{equation}\label{eq: clambake}
a_m \to \infty \text{ and } \mathbb{P}\left( \cap_{j=1}^m \Phi_{a_m}^{-1} (A_{j,n(j)}) \right) \geq 1-\e/2 \text{ for all } m\ .
\end{equation}
We do this by induction on $m$. For $m=1$, let $a_1$ be any number in the set $\Lambda_{1,K_0(1)}$. By definition then $\mathbb{P}(\Phi_{a_1}^{-1}(A_{1,n(1)})) \geq 1-\e/2$. Assuming that we have fixed $a_1, \ldots, a_m$, we now define $a_{m+1}$. Let $k$ be such that $k \geq \max \{K_0(1), \ldots, K_0(m+1)\}$ and $n_k \geq 3a_m$ and consider $\Lambda_{1,k}, \ldots, \Lambda_{m+1, k}$ as above. The intersection of these sets has Lebesgue measure at least $3n_k/4$ so choose $a_{m+1}$ as any element of the nonempty set $(3a_m/2,n_k] \cap \left[ \cap_{i=1}^{m+1} \Lambda_{i,k}\right]$. For this choice,
\[
1- \mathbb{P}\left(\cap_{j=1}^{m+1} \Phi_{a_{m+1}}^{-1}(A_{j,n(j)}) \right) \leq \sum_{j=1}^\infty \e /2^{j+2} = \e/4\ .
\]
As $a_{m+1} \geq 3a_m/2$, the condition $a_m \to \infty$ holds and we are done proving \eqref{eq: clambake}.

From \eqref{eq: clambake}, we deduce $\mathbb{P}(A) \geq 1-\e/2$, where
\[
A = \{\cap_{j=1}^m \Phi_{a_m}^{-1}(A_{j,n(j)}) \text{ occurs for infinitely many } m\}\ .
\]
We complete the proof by showing that the statement of the theorem holds for any $\omega\in A$. Fix such an $\omega$ and a random subsequence $(a_{m_k})$ of $(a_m)$ such that $\omega \in \cap_{j=1}^{m_k}\Phi_{a_{m_k}}^{-1}(A_{j,n(j)})$ for all $k$. By extracting a further subsequence, we may assume that $\mathbb{G}_{L_{a_{m_k}}(\varpi)}$ converges to some graph $G$. The event $\Phi_{\alpha}^{-1}(A_{j,n(j)})$ is exactly that the graph $\mathbb{G}_{L_\alpha(\varpi)}$ satisfies the conditions of $A_{j,n(j)}$ above, so in particular, it has no undirected circuits in $[-j,j]^2$, all directed paths starting in $[-j,j]^2$ coalesce before leaving $[-n(j),n(j)]^2$, no directed paths connect $[-n(j),n(j)]^2 \setminus (-n(j),n(j))^2$ to $[-j,j]^2$, and all vertices in $[-j,j]^2$ have one forward neighbor in $[-n(j),n(j)]^2$. On the subsequence $(a_{m_k})$, the events $\Phi_{a_{m_k}}^{-1}(A_{1,n(1)})$ occur for all $k$, so $G$ must satisfy the conditions of $A_{1,n(1)}$ as well. The same is true for $A_{j,n(j)}$ for all $j$, so $G$ satisfies the conditions of the theorem.

\subsection{Proof of Theorem~\ref{thm: exceptional_set}}

This theorem follows directly from results of the previous sections. Assume either {\bf A1'} or both {\bf A2'} and the upward finite energy property. For the first part of the theorem, suppose that $\partial \mathcal{B}$ is differentiable at $v_\theta$. Choose $\varpi = v_\theta$ and construct the measure $\mu$ as in Section~\ref{sec: mudef}. Given $(\omega,\Theta,\eta) \in \widetilde \Omega$, let $\mathbb{G}(\eta)$ be the geodesic graph associated to $\eta$. By Theorems~\ref{thm: nachostheorem}, \ref{thm: Gcoalescethm} and \ref{no_back_path}, with $\mu$-probability one, all directed paths in $\mathbb{G}$ are asymptotically directed in $I_\theta$, they coalesce, and no vertex $x$ has $|C_x|$ infinite. Call this event $A$ and define
\[
\hat \Omega = \{\omega \in \Omega_1 : \mu(A \mid \omega) = 1\}\ .
\]
$\mu(\cdot \mid \omega)$ is the regular conditional probability measure. $\hat \Omega$ is a measurable set and satisfies $\mathbb{P}(\hat \Omega)=1$ since the marginal of $\mu$ on $\Omega_1$ is $\mathbb{P}$. Further, for each $\omega \in \hat \Omega$, the theorem holds.

For the other two parts of the theorem we simply argue as in the proof of Corollaries~\ref{cor: exposed} and \ref{cor: extreme}. In the former case we just notice that if $v_\theta$ is also exposed, then $I_\theta = \{\theta\}$. In the latter case, we find a point $v_\theta$ on the arc joining $v_{\theta_1}$ to $v_{\theta_2}$ at which $\partial \mathcal{B}$ is differentiable. The set $I_\theta$ contains only angles associated to points on the arc and we are done.

\appendix

\section{Measurability of $\alpha \mapsto \mu_\alpha(A)$.}\label{sec: appendix}

In this section we show that for all Borel measurable $A \subseteq \widetilde \Omega$, $\alpha \mapsto \mu_\alpha(A)$ is Lebesgue measurable. By the monotone class theorem, it suffices to consider the case that $A$ is a cylinder event; that is, that there exists $M>0$ such that $A$ depends only on passage times $\omega_e$, Busemann increments $(\theta_1(v),\theta_2(v))$ and graph variables $\eta(f)$ for vertices $v$ in $[-M,M]^2$, and edges $e$ and directed edges $f$ with both endpoints in $[-M,M]^2$. Recall that for $\alpha \in \mathbb{R}$,
\[
\hat L_\alpha = \{x \in \mathbb{Z}^2 : x + [-1/2,1/2)^2 \cap L_\alpha \neq \varnothing\}
\]
and that passage times to $L_\alpha$ are actually defined to $\hat L_\alpha$. We are interested in how this set changes near $[-M,M]^2$ as we vary $\alpha$. For this reason, define for each $v \in \mathbb{Z}^2$
\[
C_v^- = \inf \{ \alpha: v \in \hat L_\alpha \} \text{ and } C_v^+ = \sup \{ \alpha : v \in \hat L_\alpha\}\ .
\]
It follows that for all $v$, $C_v^- < C_v^+$ and
\[
v \in \begin{cases}
\hat L_\alpha & \text{ if } \alpha \in (C_v^-, C_v^+) \\
\hat L_\alpha^c & \text{ if } \alpha \in \mathbb{R} \setminus [C_v^-, C_v^+]
\end{cases}\ .
\]
Define the set
\[
X = \cup_{v \in \mathbb{Z}^2} \{C_v^-, C_v^+\}
\]
and note that $X$ is countable. To prove Lebesgue measurability of $\alpha \mapsto \mu_\alpha(A)$, we show that
\begin{equation}\label{eq: it_suffices}
f(\alpha) := \mu_\alpha(A) \text{ is continuous except at } \alpha \in X\ .
\end{equation}

Let $\alpha \in [0,n] \setminus X$ and let $\e>0$. For any integer $N \geq M$ such that $[-N,N]^2$ intersects $\hat L_\alpha$ let $\mathcal{P}_N$ be the collection of all lattices paths whose vertices are in $[-N,N]^2$. Last define the approximate passage times for $x \in [-N,N]^2$
\[
\tau_N(x,L_\alpha) = \min_{\stackrel{x \in \gamma \in \mathcal{P}_N}{\gamma \cap \hat L_\alpha \neq \varnothing}} \tau(\gamma)
\]
and geodesics $G_N(x,L_\alpha)$ to be the minimizing paths. Let $G(x,L_\alpha)$ be the original geodesic from $x$ to $L_\alpha$. Using the shape theorem, we can choose $N$ large enough that 
\begin{equation}\label{eq: Nfix}
\mathbb{P}\left( \min_{\stackrel{v \in [-M,M]^2}{w \notin (-N,N)^2}} \tau(v,w) > \max_{v \in [-M,M]^2} \tau(v,L_\alpha) \right) \geq 1-\e\ .
\end{equation}

For $N$ fixed as above, the condition that $\alpha \notin X$ implies that we can choose $\delta>0$ such that the interval $(\alpha-\delta, \alpha + \delta)$ is contained in the complement of the finite set
\[
X_N = \cup_{v \in [-N,N]^2} \{C_v^-,C_v^+\}\ .
\]
It follows that 
\begin{equation}\label{eq: deltafix}
\text{for all }\beta \text{ with }|\alpha - \beta| < \delta,~ \hat L_\alpha \cap [-N,N]^2 = \hat L_\beta \cap [-N,N]^2\ .
\end{equation}

Having fixed $\delta$ above we now prove that if $|\beta - \alpha| < \delta$ then $|\mu_\alpha(A) - \mu_\beta(A)| < \e$. Using the definition of $\Phi_\alpha$ we can first give an upper bound 
\begin{equation}\label{eq: pizzaend}
|\mu_\alpha(A) - \mu_\beta(A)| \leq \mathbb{P}(\Phi_\alpha^{-1}(A) \Delta \Phi_\beta^{-1}(A))\ ,
\end{equation}
 where $\Delta$ is the symmetric difference operator. Note that the events on the right side are determined by (a) $\omega_e$ for $e$ with both endpoints in $[-M,M]^2$, (b) the geodesics $G(x,L_\alpha)$ and $G(x,L_\beta)$ from all points $x \in [-M,M]^2$ to the lines $L_\alpha$ and $L_\beta$ and (c) the passage times of these geodesics. Therefore the right side of \eqref{eq: pizzaend} is bounded above by
\[
\mathbb{P}( \exists ~x \in [-M,M]^2 \text{ such that } G(x,L_\alpha) \neq G(x,L_\beta))\ .
\]
However if such an $x$ exists then by \eqref{eq: deltafix}, one of two geodesics must exit the box $[-N,N]^2$. A subpath of this geodesic must cross from $[-M,M]^2$ to the complement of $(-N,N)^2$, so the event $E(M,N)$ in \eqref{eq: Nfix} cannot occur. Thus
\[
|\mu_\alpha(A) - \mu_\beta(A)| \leq \mathbb{P}(E(M,N)^c) < \e \text{ if } |\beta - \alpha| < \delta\ ,
\]
so $f$ is continuous at $\alpha$, giving measurability of $f$.

\bigskip
\noindent
{\bf Acknowledgements.}
The authors thank the organizers of the PASI school in Santiago and Buenos Aires, where some of this work was done. M. D. thanks C. Newman and the Courant Institute for summer funds and support. J. H. thanks M. Aizenman for funds and advising.

\bigskip
\noindent
{\tt Michael Damron: mdamron@math.princeton.edu \\ Jack Hanson: ~~~jthanson@princeton.edu}

\end{document}